\documentclass{gtpart}

\usepackage{pinlabel}
\usepackage{amsmath,amssymb,rotating}
\usepackage{stmaryrd,flafter}
\usepackage{diagrams}

\title{Holomorphic Discs and Surgery Exact Triangles} 

\author{Bijan Sahamie} 
\givenname{Bijan}
\surname{Sahamie}
\address{Mathematisches Institut der LMU M\"{u}nchen, 
Theresienstrasse 39, 80333 M\"{u}nchen}
\email{sahamie@math.lmu.de}
\urladdr{http://www.math.lmu.de/~sahamie}


\theoremstyle{plain} 
\newtheorem{theorem}{Theorem}[section]   
   
\newtheorem{lem}[theorem]{Lemma}         
\newtheorem{prop}[theorem]{Proposition}
\newtheorem{cor}[theorem]{Corollary}
\newtheorem{q}{Question}

\theoremstyle{definition}
\newtheorem{definition}[theorem]{Definition}

\def\eqref#1{(\ref{#1})}
\numberwithin{equation}{section}

\newarrow{Mto}{|}{-}{-}{-}{>}
\newarrow{Into}{C}{-}{-}{-}{>}
\newarrow{Onto}{-}{-}{-}{-}{>>}

\begin{document}
%
\newcommand{\sone}{\mathbb{S}^1}
\newcommand{\lra}{\longrightarrow}
\newcommand{\lmt}{\longmapsto}
%

%
%
\newcommand{\ztwo}{\mathbb{Z}_2}
\newcommand{\RP}{\mbox{\rm RP}}
\newcommand{\N}{\mathbb{N}}
\newcommand{\spinc}{\mbox{\rm Spin}^c}
\newcommand{\spiny}{\mbox{\rm Spin}^c_3}
\newcommand{\homology}{\mathcal{H}}
\newcommand{\phitilde}{\widetilde{\phi}}
\newcommand{\phibar}{\overline{\phi}}
\newcommand{\grading}{\mbox{\rm gr}}
\newcommand{\riemop}{\partial_{\com_s}}
\newcommand{\banbund}{\mathcal{B}}
\newcommand{\chat}{\hat{c}}
\newcommand{\alphafat}{\boldsymbol\alpha}
\newcommand{\betafat}{\boldsymbol\beta}
\newcommand{\deltafat}{\boldsymbol\delta}
\newcommand{\betaprimefat}{\boldsymbol\beta'}
\newcommand{\mA}{\mathcal{A}}
\newcommand{\mJ}{\mathcal{J}}
\newcommand{\mM}{\mathcal{M}}
\newcommand{\mN}{\mathcal{N}}
\newcommand{\mH}{\mathcal{H}}
\newcommand{\shb}{\mbox{\rm \underline{h}}}
\newcommand{\fraks}{\mathfrak{s}}
\newcommand{\frakt}{\mathfrak{t}}

\def\co{\colon\thinspace}
\newcommand{\stwo}{\mathbb{S}^2}
\newcommand{\pd}{\text{PD}}
\newcommand{\sprung}{\\[0.3cm]}
\newcommand{\fund}{\pi_1}
\newcommand{\kerg}{\mbox{\rm Ker}_G\,}
\def\Ker#1{\mbox{\rm Ker}_{#1}\,}
\newcommand{\im}{\mbox{\rm im}\,}
\newcommand{\id}{\mbox{\rm id}}
\newcommand{\sothree}{\mathbb{SO}_3}
\newcommand{\sthree}{\mathbb{S}^{3}}
\newcommand{\disc}{\mbox{\rm D}}
\newcommand{\systwo}{C^{\infty}(Y,\stwo)}
\newcommand{\inner}{\text{int}}
\newcommand{\bk}{\backslash}

%
%
\newcommand{\contstand}{(\mathbb{R}^3,\xi_0)}
\newcommand{\cont}{(M,\xi)}
\newcommand{\sym}{\xi_{sym}}
\newcommand{\xistd}{\xi_{std}}
\newcommand{\cyl}{\mbox{\rm Cyl}_{r_0}^\mu}
\newcommand{\stap}{\mbox{\rm S}_+}
\newcommand{\stam}{\mbox{\rm S}_-}
\newcommand{\stapm}{\mbox{\rm S}_\pm}
\newcommand{\modulo}{\;\;\mbox{\rm mod}\,}

%
%
\newcommand{\crit}{\mbox{\rm Crit}}
\newcommand{\MC}{\mbox{\rm MC}}
\newcommand{\bmorse}{\partial^{\mbox{\rm \begin{tiny}M\!C\end{tiny}}}}
\newcommand{\M}{\mathcal{M}}
\newcommand{\stable}{\mbox{\rm W}^s}
\newcommand{\unstable}{\mbox{\rm W}^u}
\newcommand{\ind}{\mbox{\rm ind}}
\newcommand{\Mhat}{\widehat{\M}}
\newcommand{\modphi}{\mathcal{M}_\phi}
%
%
%

%
%
\newcommand{\com}{\mathcal{J}}
\newcommand{\moduli}{\mathcal{M}_{\mathcal{J}_s}(x,y)}
\newcommand{\modulit}{\mathcal{M}_{J_{s,t}}}
\newcommand{\modulittau}{\mathcal{M}_{\com_{s,t}(\tau)}}
\newcommand{\modulittauphi}{\mathcal{M}_{\com_{s,t}(\tau),\phi}}
\newcommand{\modulittaubig}{\mathcal{M}_{\com_{s,t}(\tau)}}
\newcommand{\modulittaubigphi}{\mathcal{M}_{\com_{s,t}(\tau),\phi}}
\newcommand{\modhat}{\widehat{\mathcal{M}}_{J_s}(x,y)}
\newcommand{\modhatxy}{\widehat{\mathcal{M}}(x,y)}
\newcommand{\modhatyw}{\widehat{\mathcal{M}}(y,w)}
\newcommand{\modhatone}{\widehat{\mathcal{M}}_{J_{s,1}}}
\newcommand{\modhatzero}{\widehat{\mathcal{M}}_{J_{s,0}}}
\newcommand{\modhatphi}{\widehat{\mathcal{M}}_{\phi}}
\newcommand{\modhatphib}{\widehat{\mathcal{M}}_{[\phi]}}
\newcommand{\moduliiso}{\mathcal{M}^{t}}
\newcommand{\modspace}{\mathcal{M}}
\newcommand{\modfamily}{\mathcal{M}_{\com_{s,t}}}
\newcommand{\modfamilyphi}{\mathcal{M}_{\com_{s,t},\phi}}
\newcommand{\modtriangle}{\mathcal{M}^{\Delta}}

\newcommand{\fglue}{f_{\mbox{\rm {\tiny glue}}}}
\newcommand{\phihat}{\widehat{\phi}}
\newcommand{\Phihat}{\widehat{\Phi}}
\newcommand{\Psihat}{\widehat{\Psi}}
\newcommand{\Phiinfty}{\Phi^\infty}
\newcommand{\Hhat}{\widehat{H}}
\newcommand{\Dhat}{\widehat{\D}}
\newcommand{\cops}{\partial_{J_s}}
\newcommand{\talpha}{\mathbb{T}_{\boldsymbol\alpha}}
\newcommand{\tbeta}{\mathbb{T}_{\boldsymbol\beta}}
\newcommand{\tgamma}{\mathbb{T}_{\boldsymbol\gamma}}
\def\marge#1{\marginpar{\scriptsize{#1}}}
\def\br#1{\begin{rotate}{90}#1\end{rotate}}
\newcommand{\hfhat}{\widehat{\mbox{\rm HF}}}
\newcommand{\sfh}{\mbox{\rm SFH}}
\newcommand{\sbottom}{\underline{s}}
\newcommand{\tbottom}{\underline{t}}
\newcommand{\cfhat}{\widehat{\mbox{\rm CF}}}
\newcommand{\cfinfty}{\mbox{\rm CF}^\infty}
\def\cfbb#1{\mbox{\rm CF}^+_{\leq #1}}
\def\cfbo#1{\mbox{\rm CF}^-_{\geq -#1}}
\newcommand{\cfleq}{\mbox{\rm CF}^{\leq 0}}
\newcommand{\cfcirc}{\mbox{\rm CF}^\circ}
\newcommand{\cfkinfty}{\mbox{\rm CFK}^{\infty}}
\newcommand{\cfkhat}{\widehat{\mbox{\rm CFK}}}
\newcommand{\cfkminus}{\mbox{\rm CFK}^{-}}
\newcommand{\cfkpstar}{\mbox{\rm CFK}^{+,*}}
\newcommand{\cfkostar}{\mbox{\rm CFK}^{0,*}}
\newcommand{\hfkcirc}{\mbox{\rm HFK}^\circ}
\newcommand{\hfkhat}{\widehat{\mbox{\rm HFK}}}
\newcommand{\hfkplus}{\mbox{\rm HFK}^+}
\newcommand{\hfkminus}{\mbox{\rm HFK}^-}
\newcommand{\hfkinfty}{\mbox{\rm HFK}^\infty}
\def\cfinftyfilt#1#2{\mbox{\rm CFK}^{#1,#2}}
\newcommand{\hfinfty}{\mbox{\rm HF}^\infty}
\newcommand{\hfinftwist}{\underline{\mbox{\rm {HF}}}^\infty}
\newcommand{\fhat}{\widehat{f}}
\newcommand{\fcirc}{f^\circ}
\newcommand{\Fhat}{\widehat{F}}
\newcommand{\Fcirc}{F^\circ}
\newcommand{\hattheta}{\widehat{\Theta}}
\newcommand{\shattheta}{\widehat{\theta}}
\newcommand{\cfminus}{\mbox{\rm CF}^-}
\newcommand{\hfminus}{\mbox{\rm HF}^-}
\newcommand{\cfplus}{\mbox{\rm CF}^+}
\newcommand{\hfplus}{\mbox{\rm HF}^+}
\newcommand{\hfcirc}{\mbox{\rm HF}^\circ}
\def\hfbb#1{\hfplus_{\leq #1}}
\def\hfbo#1{\hfminus_{\geq -#1}}
\newcommand{\Hs}{\mathcal{H}_s}
\newcommand{\gr}{\mbox{gr}}
\newcommand{\parinfty}{\partial^\infty}
\newcommand{\parhat}{\widehat{\partial}}
\newcommand{\parplus}{\partial^+}
\newcommand{\parminus}{\partial^-}
\newcommand{\symg}{\mbox{\rm Sym}^g(\Sigma)}
\newcommand{\symgg}{\mbox{\rm Sym}^{2g}(\Sigma)}
\newcommand{\symc}{\mbox{\rm Sym}^g(\mathbb{C})}
\newcommand{\pitwo}{\pi_2}
\newcommand{\pitwoham}{\pi_2^{t}}
\newcommand{\symcon}{\mbox{\rm Sym}^g(\Sigma_1\#\Sigma_2)}
\newcommand{\symgone}{\mbox{\rm Sym}^{g_1}(\Sigma_1)}
\newcommand{\symgtwo}{\mbox{\rm Sym}^{g_2}(\Sigma_2)}
\newcommand{\symgmo}{\mbox{\rm Sym}^{g-1}(\Sigma)}
\newcommand{\symggmo}{\mbox{\rm Sym}^{2g-1}(\Sigma)}
\newcommand{\dom}{\mathcal{D}}
\newcommand{\bigtrans}{\left.\bigcap\hspace{-0.27cm}\right|\hspace{0.1cm}}
\newcommand{\tlt}{\times\ldots\times}
\newcommand{\Isotopy}{\widehat{\Gamma}_{\Psi_t}}
\newcommand{\Isotopyinverse}{\widehat{\Gamma}_{\Psi_{1-t}}}
\newcommand{\orient}{\mathnormal{o}}
\newcommand{\ob}{\mathnormal{ob}}
\newcommand{\SL}{\mbox{\rm SL}}
\newcommand{\rhotilde}{\widetilde{\rho}}
\newcommand{\domstar}{\dom_*}
\newcommand{\domststar}{\dom_{**}}
\newcommand{\betaprime}{\beta'}
\newcommand{\afat}{{\boldsymbol\alpha}}
\newcommand{\bfat}{{\boldsymbol\beta}}
\newcommand{\gfat}{{\boldsymbol\gamma}}
\newcommand{\dfat}{{\boldsymbol\delta}}
\newcommand{\xfat}{{\boldsymbol x}}
\newcommand{\yfat}{{\boldsymbol y}}
\newcommand{\zfat}{{\boldsymbol z}}
\newcommand{\bfatprime}{{\boldsymbol\beta}'}
\newcommand{\dfattilde}{\widetilde{{\boldsymbol\delta}}}
\newcommand{\betapp}{\beta''}
\newcommand{\betatilde}{\widetilde{\beta}}
\newcommand{\bfattilde}{\widetilde{{\boldsymbol\beta}}}
\newcommand{\deltaprime}{\delta'}
\newcommand{\tbetaprime}{\mathbb{T}_{\boldsymbol\beta'}}
\newcommand{\talphaprime}{\mathbb{T}_{\boldsymbol\alpha'}}
\newcommand{\tdelta}{\mathbb{T}_{\boldsymbol\delta}}
\newcommand{\phidelta}{\phi^{\Delta}}
\newcommand{\domtilde}{\widetilde{\dom}}
\newcommand{\loss}{\widehat{\mathcal{L}}}
\newcommand{\bargamma}{\overline{\Gamma}}
\newcommand{\alphaprime}{\alpha'}
\newcommand{\ga}{\Gamma_{\alpha;\beta',\beta''}}
\newcommand{\gbone}{\Gamma_{\alpha;\beta,\widetilde{\beta}}^{w,1}}
\newcommand{\gbtwo}{\Gamma_{\alpha;\beta,\widetilde{\beta}}^{w,2}}
\newcommand{\gbthree}{\Gamma_{\alpha;\beta,\widetilde{\beta}}^{w,3}}
\newcommand{\gbfour}{\Gamma_{\alpha;\beta,\widetilde{\beta}}^{w,4}}
\newcommand{\gcone}{\Gamma_{\alpha;\delta,\delta'}^{w,1}}
\newcommand{\gctwo}{\Gamma_{\alpha;\delta,\delta'}^{w,2}}
\newcommand{\gcthree}{\Gamma_{\alpha;\delta,\delta'}^{w,3}}
\newcommand{\gcfour}{\Gamma_{\alpha;\delta,\delta'}^{w,4}}
\newcommand{\xpi}{x^+_i}
\newcommand{\xmi}{x^-_i}
\newcommand{\xp}{x^+}
\newcommand{\xm}{x^-}
\newcommand{\deltatilde}{\widetilde{\delta}}
\newcommand{\tbetatilde}{\mathbb{T}_{\widetilde{\boldsymbol\beta}}}

\newcommand{\eab}{\epsilon_{\alpha\beta}}
\newcommand{\ead}{\epsilon_{\alpha\delta}}
\newcommand{\hqhat}{\widehat{\mbox{\rm HQ}}}
\newcommand{\cupb}{\cup_\partial}
\newcommand{\oa}{\overline{a}}
\newcommand{\ab}{{\boldsymbol\alpha\boldsymbol\beta}}
\newcommand{\ad}{{\boldsymbol\alpha\boldsymbol\delta}}
\newcommand{\adb}{{\boldsymbol\alpha\boldsymbol\delta\boldsymbol\beta}}
\newcommand{\tila}{\widetilde{a}}
\newcommand{\tilb}{\widetilde{b}}
\def\pdehn#1#2{D_{#1}^{+,#2}} 
\def\ndehn#1#2{D_{#1}^{-,#2}}

%
%
\newcommand{\bund}{\mathcal{P}}
\newcommand{\diag}{\Delta^{\!\!E}}
\newcommand{\inter}{m_{\diag}}
\newcommand{\ozs}{Ozsv\'{a}th}
\newcommand{\sza}{Szab\'{o}}
\begin{abstract}
We show a connection between a surgery exact sequence
in knot Floer homology and the sequence derived in \cite{Saha01}.
As a consequence of this relationship we see that the exact sequence
in \cite{Saha01} also works with coherent orientations and admits
refinements with respect to $\spinc$-structures. As an application 
of this discussion, we prove that the ranks of the image and kernel 
of certain cobordism maps between knot Floer homologies can be 
computed combinatorially by relating them to a count of certain moduli
spaces of holomorphic disks.
\end{abstract}
\maketitle
\section{Introduction}\label{parone}
\noindent Heegaard Floer homology was introduced by Peter Ozsv\'ath and 
Zoltan Szab\'o in \cite{OsZa01} (see~\cite{Saha02} for a detailed 
introduction) and has turned out to be a useful tool in the study 
of low-dimensional topology. They also defined variants of this 
homology theory which are topological invariants of a pair $(Y,K)$ 
where $Y$ is a closed, oriented $3$-manifold and $K\subset Y$ a 
null-homologous knot (see \cite{OsZa04}). One of the main features of 
this homology theory is the existence of exact sequences which serve 
as a main tool in calculations. The exact sequences Ozsv\'{a}th and 
Szab\'{o} provided either just contained Heegaard Floer homology groups 
or they just contained knot Floer homology groups. 
\begin{theorem}[Theorem~8.2.~of~\cite{OsZa04}, 
cf.~Theorem~2.7.~of~\cite{OsSti}]\label{oszathm}
Let $Y$ be a closed, oriented $3$-manifold with framed knot $C\subset Y$ 
with framing $n$ and let $K\subset Y\backslash C$ be a null-homologous 
knot. Denote by $Y_n(C)$ the result of a surgery along $C$ in $Y$ with 
framing $n$ and denote by $Y_{n+1}(C)$ the result of a $(-1)$-surgery 
along a meridional curve of $C$ in $Y_n(C)$. Finally, we can interpret 
$Y$ itself as a result of a surgery along a framed knot in 
$Y_{n+1}(C)$. Then each surgery can be described by a Heegaard triple 
diagram which induces a map defined by counting holomorphic triangles. 
These fit into the following sequence
\begin{diagram}[size=2em,labelstyle=\scriptstyle] \dots&
\rTo^{\Fhat_3}&
\hfkhat(Y,K)& 
&
\rTo^{\Fhat_{1}}&
\hfkhat(Y_{n}(C),K)&
\rTo^{\Fhat_{2}}& 
&
\hfkhat(Y_{n+1}(C),K)&
\rTo^{\Fhat_3}& 
\dots \\
\end{diagram}
which is exact.
\end{theorem}
In \cite{Saha01} we introduced a new exact sequence which, in contrast
to the known sequences, contains both Heegaard Floer homology groups 
and knot Floer homology groups: We proved that the following
sequence is exact (cf.~Sequence~\eqref{dtses01})
\begin{equation}
\begin{diagram}[size=2em,labelstyle=\scriptstyle]
 \dots & \rTo^{f_*} &\hfkhat(Y,K)&&\rTo^{\Gamma_1}&&
 \hfhat(Y_{-1}(K))&&\rTo^{\Gamma_2}&&\hfkhat(Y_0(K),\mu)&
 \rTo^{f_*}&\dots
\end{diagram}\label{eq:aha}
\end{equation}
where $\mu$ is a meridian of $K$, interpreted as sitting in $Y_0(K)$.
The maps $\Gamma_1$ and $\Gamma_2$ are not defined by counting holomorphic
triangles and, in fact, the map $f_*$ is given by counting holomorphic disks
in a suitable Heegaard diagram. Comparing both sequences, 
i.e.~Sequence~\eqref{eq:aha} and the sequence given in 
Theorem~\ref{oszathm}, we see that in the Dehn twist sequence, $K$ 
serves as the knot and the surgery curve. This is a situation which in 
general violates the assumption in Theorem~\ref{oszathm} that the knot 
is null-homologous in the complement of the surgery curve. We think that 
it is natural to pose the following question.
\begin{q}\label{q:q} Is it possible to relate the Dehn twist sequence 
with the sequence given in Theorem~\ref{oszathm} or with one which 
is defined, similarly?
\end{q}
Providing a discussion and a possible answer to this question will be the 
main goal of this article.
To do that we will start proving the following statement.
\begin{prop}\label{firsttriangle} There is an exact sequence
\begin{equation}
\begin{diagram}[size=2em,labelstyle=\scriptstyle]
\dots&
\rTo^{\partial_*}&
\hfkhat(Y,K)& 
&
\rTo^{\Fhat_{1}}&
\hfhat(Y_{-1}(K))&
\rTo^{\Fhat_{2}}& 
&
\hfkhat(Y_0(K),\mu)&
\rTo^{\partial_*}& 
\dots \\
\end{diagram}\label{knotses02}
\end{equation}
where the $\Fhat_i$, $i=1,2$, are maps defined by counting holomorphic 
triangles in suitable doubly-pointed Heegaard triple diagrams, and 
$\partial_*$ is a connecting morphism.
\end{prop}
This result is set up using different techniques than 
Ozsv\'{a}th and Szab\'{o} utilized for the surgery exact sequence in 
knot Floer homology. It is based on bringing the attaching circles of the
underlying Heegaard triple diagrams into an opportune position 
(see~Figure~\ref{Fig:atcirc}) and then
providing a careful analysis of the underlying doubly-pointed Heegaard 
triple diagrams (see Lemma~\ref{lemSES} and Lemma~\ref{lemSES2}). Although 
not essential, it is opportune to work with Heegaard diagrams that are induced 
by open books. In this particular situation the analysis of the Heegaard 
diagrams and of the domains of Whitney triangles is easier. 
Comparing this sequence with Theorem~\ref{oszathm}, notice, that we do not 
impose any condition on $K$, 
i.e.~the knot may be homologically essential. Furthermore, notice, that the third
map in the sequence, $\partial_*$, is a connecting morphism which is not given
by counting holomorphic triangles. As a side-effect of our analysis we will 
see that the Sequences \eqref{knotses02} and \eqref{eq:aha} {\it interact}
in a commutative diagram.
\begin{theorem}\label{ThdiagCOM} Let $Y$ be a closed, oriented $3$-manifold
and $K\subset Y$ a knot. Denote by $Y_{-1}(K)$ (resp.~$Y_{0}(K)$) the 
result of performing a $(-1)$-surgery (resp.~$0$-surgery) along $K$. We denote by $\mu$
a meridian of $K$. Then, all triangles and boxes in the following diagram
commute.
\begin{equation}
\begin{diagram}[size=2em,labelstyle=\scriptstyle]
& & & & & & & & 
\hfkhat(Y_0(K),\mu) & \rTo^{f_*}
& \dots\\ 
& & & & & & 
\ruTo^{\Gamma_2} & 
& 
\dTo^{\Fhat_{4}}_{\cong}
& &\uTo \\ 
\dots&
\rTo^{\partial_*}&
\hfkhat(Y,K)& 
&
\rTo^{\Fhat_{1}}&
\hfhat(Y_{-1}(K))&
\rTo^{\Fhat_{2}}& 
&
\hfkhat(Y_0(K),\mu)&
\rTo^{\partial_*}& 
\dots \\
\uTo& & 
\dTo^{\Fhat_{5}}_{\cong}& 
& 
\ruTo^{\Gamma_1} &       
& & & & & \\       
\dots&                 
\rTo^{f_*}& 
\hfkhat(Y,K) &                 
& & & & & & & 
\end{diagram}\label{diagCOM}
\end{equation}
Here, the isomorphisms $\Fhat_4$ and $\Fhat_5$ are also defined by 
counting holomorphic triangles in suitable doubly-pointed Heegaard 
triple diagrams.
\end{theorem}
The horizontal sequence is, in fact, the Sequence~\eqref{knotses02} and the
diagonal sequence is the Dehn twist sequence. Thus, we may interpret this
theorem as a possible answer to Question \ref{q:q}. This result, indeed,
has some implications we would like to discuss. 
\vspace{0.3cm}\\
{\bf Implication I.} 
For a contact manifold $(Y,\xi)$ Ozsv\'{a}th and
Szab\'{o} introduced a contact invariant $\chat(\xi)$, called the
contact element, which is an element of 
$\hfhat(-Y,\fraks_\xi)$ where $\fraks_\xi$ is the $\spinc$-structure associated
to $\xi$ (see~\cite{OsZa02} and cf.~\cite{HKM}). Furthermore, Lisca, Ozsv\'{a}th, Stipsicz and Szab\'{o} introduced in \cite{LOSS} an invariant $\loss(L)$ of 
a Legendrian knot $L$ which sits in $\hfkhat(-Y,L;\fraks_\xi)$. The contact 
element has proved to be a powerful obstruction to overtwistedness (see~\cite{LiSti01,LiSti02,LiSti03}). Being 
interested in contact geometry, to us, besides Question~\ref{q:q}, it was 
interesting to know if the Dehn twist sequences or the maps involved in 
these sequences can be defined with coherent orientations (and, hence, with $\Z$-coefficients) and if there is a refined version (with respect to $\spinc$-structures). Our interest in this question originates 
from \cite[Theorem 6.1]{Saha01}.
\begin{theorem}[Theorem 6.1~of~\cite{Saha01}]\label{myold} Let $(Y,\xi)$ 
be a contact manifold and $L$ an oriented Legendrian knot. Let $W$ be 
the cobordism induced by $(+1)$-contact surgery along $L$ and denote 
by $(Y_L,\xi_L)$ the contact manifold we obtain from $(Y,\xi)$ by 
performing this surgery. Then, the cobordism $-W$ induces a map
\[
 \Gamma_{-W}
 \co
 \hfkhat(-Y,L)
 \lra
 \hfhat(-Y_{+1}(L))
\]
such that $\Gamma_{-W}(\loss(L))=\chat(\xi_L)$.
\end{theorem}
There exists a similar naturality property that connects the contact elements 
before and after the contact surgery (see~\cite[Theorem 2.3]{LiSti04}). In the calculations and applications given by
Lisca and Stipsicz (see~\cite{LiSti01,LiSti02,LiSti03,LiSti04}) this naturality 
property was one of the main calculational
tools. Additionally, in their work, the understanding of the map providing 
the naturality was a major ingredient. In light of this, it is natural to ask 
for refinements and for coherent orientations of $\Gamma_{-W}$.\vspace{0.3cm}\\
As a matter of fact, Theorem~\ref{ThdiagCOM} may be applied to
introduce coherent orientations and refinements into the Dehn twist sequence: 
The maps $\Fhat_i$ all admit refinements (see~\S\ref{cobmapintro}) and, as such, 
we are able to refine 
the Sequence~\eqref{knotses02}. So, we may apply the commutative diagram given
in Theorem~\ref{ThdiagCOM} to provide a refined version of the Dehn twist
sequence. The same strategy may be applied to bring coherent orientations
into the Dehn twist sequence. We outline this at the end of 
Section~\ref{sec:setadts}. We summarize this briefly with the following corollary.
\begin{cor}\label{bla} The Dehn twist Sequences~$(\ref{dtses01})$ can be 
defined with coherent orientations. Furthermore, these sequences refine 
with respect to $\spinc$-structures.\hfill$\square$
\end{cor}
The element $\chat(\xi_L)$ is the element in homology, induced by a special
generator $eh_{\chat}$ of $\cfhat(-Y_+(L))$ one can specify (see~\cite{HKM}). Theorem~\ref{myold} is proved by identifying a generator $eh_{L}$ (in fact, the
element for which $[eh_L]=\loss(L)$) of $\cfkhat(-Y,L)$ which is mapped onto $eh_{\chat}$ under $\Gamma_{-W}$. The fact that the element 
$[eh_{\chat}]=\chat(\xi_L)$ is invariant under all choices made in its 
definition (see~\cite{HKM}) together with the invariance properties we 
proved in \cite[\S3]{Saha02} for $\Gamma_{-W}$ may be assembled to an alternative 
proof of the fact that $[eh_{L}]$ does not depend on the choices made in its 
definition. So, 
providing coherent orientations additionally gives us evidence 
that the invariant $\loss$ can also be defined in $\Z$-coefficients and for 
Legendrian knots which are homologically essential. 
Of course, what was done in \cite{LOSS} can be slightly altered to 
provide these {\it generalizations}, as well. Because of that and since we do not 
write down this alternative approach we do not state this as a result, 
here.\vspace{0.3cm}\\
{\bf Implication II.} We prove the following statement.
\begin{theorem}\label{kerimequal}
Given a closed, oriented $3$-manifold and a knot $L\subset Y$ with framing $n$, 
then denote by $K$ a push-off of $L$ which corresponds to the $(n+1$)-framing
of $L$. Let $Y'$ be the result of a $n$-surgery along $L$ and let $K'$ be the
knot $K$ represents in $Y'$. Furthermore, 
let
\[
 \Fhat
 \co
 \hfkhat(Y,K)
 \lra
 \hfkhat(Y',K')
\]
be the map defined by counting holomorphic triangles in a suitable Heegaard 
triple diagram associated to the surgery. Then, (as part of a Dehn twist 
sequence) there is a map
\[
 f_*
 \co
 \hfkhat(Y,K)
 \lra
 \hfkhat(Y',K')
\]
which is defined by counting holomorphic disks in a suitable Heegaard diagram,
such that $\ker(\Fhat)=\ker(f_*)$ and $\im(\Fhat)=\im(f_*)$.
\end{theorem}
This result tells us that the rank of the image and the kernel of the
map $\Fhat$ can be computed using the map $f_*$. However, since $f_*$ appears
as part of a Dehn twist sequence, by its definition, it is part of a
Heegaard Floer differential (cf.~Proposition~\ref{THMTHM}). The Sarkar-Wang 
algorithm (see~\cite{sarwang}) presents a way to
combinatorially compute Heegaard Floer differentials. We will prove that
this algorithm may as well be applied in this particular situation. Thus, 
we get the following result.

\begin{prop}\label{combinmap} Given a closed, oriented $3$-manifold $Y$ and 
a framed knot $L\subset Y$ with framing $n$, then denote by $K$ a push-off of $L$ 
which corresponds to the $(n+1)$-framing of $L$. Let $Y'$ be the 
manifold obtained by performing a surgery along $L$ and let $K'$ be
the knot $K$ represents in $Y'$. In this situation we may define a map
\[
  \Fhat\co\hfkhat(Y,K)\lra\hfkhat(Y',K')
\]
by counting holomorphic triangles in a suitable Heegaard triple diagram
associated to the surgery. The rank of the kernel and image of $\Fhat$ can 
be computed combinatorially.
\end{prop}
Here, again it is opportune to work with open books: Plamenevskaya showed
in \cite{plamenev} that Heegaard diagrams induced by open books can be made
nice using the Sarkar-Wang algorithm by just using deformations that are 
induced by isotopies of the monodromy. This leads to a simplification of 
our discussion (see~proof of Proposition~\ref{combinmap}).

\paragraph{Acknowledgements.} Parts of 
the research for this paper was done during a stay at Columbia 
University. The main idea of this paper was motivated during a 
conversation with Peter Ozsv\'{a}th. The author wants to thank 
Columbia University for its hospitality and Peter Ozsv\'{a}th for 
his support. The stay was supported by the German Academic Exchange 
Service (DAAD).

\section{Preliminaries}\label{preliminaries}
\subsection{Heegaard Floer homologies}\label{prelim:01:1}
The Heegaard Floer homology group $\hfhat(Y)$ of a
$3$-manifold $Y$ was introduced in \cite{OsZa01}. The 
definition was extended for the case where $Y$ is equipped 
with a knot $K\subset Y$ to the variant $\hfkhat(Y,K)$ 
in \cite{OsZa04} (cf.~\cite{Saha01}).\vspace{0.3cm}\\
A $3$-manifold $Y$ can be described by a Heegaard diagram, which 
is a triple $(\Sigma,\afat,\bfat)$, where $\Sigma$ is an oriented 
genus-$g$ surface and $\afat=\{\alpha_1,\dots,\alpha_g\}$, 
$\bfat=\{\beta_1,\dots,\beta_g\}$ are two sets of pairwise disjoint 
simple closed curves in $\Sigma$ called {\bf attaching circles}. 
Each set of curves $\afat$ and $\bfat$ is required to consist of 
linearly independent curves in $H_1(\Sigma,\Z)$. In the following 
we will talk about the curves in the set $\afat$ (resp.~$\bfat$) 
as {\bf $\afat$-curves} (resp.~{\bf $\bfat$-curves}). Without 
loss of generality we may assume that the $\afat$-curves 
and $\bfat$-curves intersect transversely. To a Heegaard diagram 
we may associate the triple
$(\symg,\talpha,\tbeta)$ consisting of the $g$-fold symmetric power of
$\Sigma$, 
\[
  \symg=\Sigma^{\times g}/S_g,
\] 
and the submanifolds $\talpha=\alpha_1\times\dots\times\alpha_g$
and $\tbeta=\beta_1\times\dots\times\beta_g$. We define 
$\cfhat(\Sigma,\afat,\bfat)$ as the free $\ztwo$-module 
generated by the set
$\talpha\cap\tbeta$. In the following
we will just write $\cfhat$. For two intersection
points $\xfat,\yfat\in\talpha\cap\tbeta$ define 
$\pitwo(\xfat,\yfat)$ to be the set of
homology classes of {\bf Whitney discs} 
$\phi\co\disc\lra\symg$ ($\disc\subset\C$) that 
{\bf connect $\xfat$ with $\yfat$}. The map $\phi$ is called {\bf Whitney} if 
$\phi(\disc\cap\{Re<0\})\subset\talpha$ and $\phi(\disc\cap\{Re>0\})\subset\tbeta$. 
We call $\disc\cap\{Re<0\}$ the {\bf $\afat$-boundary of $\phi$} and
$\disc\cap\{Re>0\}$ the {\bf $\bfat$-boundary of $\phi$}. Such 
a Whitney disc {\bf connects $\xfat$ with $\yfat$} if 
$\phi(i)=\xfat$ and $\phi(-i)=\yfat$. 
Note that $\pitwo(\xfat,\yfat)$ can be interpreted as the subgroup of elements in
$H_2(\symg,\talpha\cup\tbeta)$ represented by discs with appropriate 
boundary conditions. We endow 
$\symg$ with a symplectic structure~$\omega$. By choosing a path of almost complex 
structures $\mJ_s$ on $\symg$ suitably (cf.~\cite{OsZa01})
all moduli spaces of holomorphic Whitney discs are Gromov-compact manifolds.
Denote by $\modphi$ the set of holomorphic Whitney discs in the equivalence
class $\phi$, and $\mu(\phi)$ the formal dimension of $\modphi$. Denote by 
$\modhatphi=\modphi/\R$ the quotient under the translation action of 
$\R$ (cf.~\cite{OsZa01}). Define $H(x,y,k)$ to be the subset of classes in
$\pitwo(\xfat,\yfat)$ that admit moduli spaces of dimension $k$. Fix a point 
$z\in\Sigma\backslash(\alpha\cup\beta)$ and define the map 
\[
  n_z\co\pitwo(\xfat,\yfat)\lra\Z,\,\phi\lmt\#(\phi,\{z\}\times\symgmo).
\] 
A boundary operator $\parhat\co\cfhat\lra\cfhat$ is given by defining it
on the generators $\xfat$ of $\cfhat$ by
\[
  \parhat\xfat
  =
  \sum_{\yfat\in\talpha\cap\tbeta}
  \sum_{\phi\in H(\xfat,\yfat,1)}
  \!\!\!\!\#\modhatphi\cdot U^{n_z(\phi)}\yfat.
\]
These homology groups are topological invariants of the manifold $Y$. 
We would like to note that not all Heegaard diagrams are suitable
for defining Heegaard Floer homology; there is an additional
condition that has to be imposed called {\bf admissibility}. 
This is a technical condition in the compactification of the moduli spaces
of holomorphic Whitney discs. A detailed knowledge of this condition 
is not important in the remainder of the present article since all
constructions are done nicely so that there will never be a 
problem. We advise the interested reader to \cite{OsZa01} .

\subsection{Knot Floer homology}\label{knotfloerhomology}
Given a knot $K\subset Y$, we can specify a certain subclass of 
Heegaard diagrams.
\begin{definition} \label{knotdiagram} A Heegaard 
diagram $(\Sigma,\afat,\bfat)$ is said to
be {\bf adapted} to the knot $K$ if $K$ is isotopic to a knot lying
in $\Sigma$ and $K$ intersects $\beta_1$ once transversely and is
disjoint from the other $\bfat$-circles.
\end{definition}
Since $K$ intersects $\beta_1$ once and is disjoint from the other 
$\bfat$-curves we know that $K$
intersects the core disc of the $2$-handle represented by $\beta_1$ once
and is disjoint from the others (after possibly isotoping the knot $K$).
Every pair $(Y,K)$ admits a Heegaard diagram adapted to $K$.
Having fixed such a Heegaard diagram $(\Sigma,\afat,\bfat)$ we can encode 
the knot $K$ in a pair of points. After isotoping $K$ onto $\Sigma$, 
we fix a small interval $I$ in $K$ containing the intersection point 
$K\cap\beta_1$. This interval should be chosen small enough such 
that $I$ does not contain any other intersections of $K$ with other 
attaching curves. The boundary $\partial I$ of $I$ determines two 
points in $\Sigma$ that lie in the complement of the attaching circles, 
i.e.~$\partial I=z-w$, where the orientation of $I$ is given by the 
knot orientation. This leads to a doubly-pointed Heegaard diagram 
$(\Sigma,\afat,\bfat,w,z)$. Conversely, a doubly-pointed Heegaard 
diagram uniquely determines a topological knot class: Connect 
$w$ with $z$ in the complement of the attaching circles $\afat$ 
and $\bfat\backslash\beta_1$ with an arc $\delta$ that crosses 
$\beta_1$, once. Connect $z$ with $w$ in the complement of $\bfat$
using an arc $\gamma$. The union $\delta\cup\gamma$ is represents the 
knot class $K$ represents. The orientation on $K$ is given by 
orienting $\delta$ such that $\partial\delta=z-w$. \vspace{0.3cm}\\
The knot chain complex $\cfkhat(Y,K)$ is the free $\ztwo$-module 
generated by the intersections $\talpha\cap\tbeta$. 
The boundary operator $\parhat^w$, for $\xfat\in\talpha\cap\tbeta$, is 
defined by
\[
  \parhat^w(\xfat)
  =
  \sum_{\yfat\in\talpha\cap\tbeta}
  \sum_{\phi\in H(\xfat,\yfat,1)}
  \#\modhatphi\cdot \yfat,
\]
where $H(\xfat,\yfat,1)\subset\pitwo(\xfat,\yfat)$ are the homotopy classes
with $\mu=1$ and $n_z=n_w=0$. We denote by $\hfkhat(Y,K)$
the associated homology theory $H_*(\cfkhat(Y,K),\parhat^w)$.

\subsection{Maps Induced by Cobordisms}\label{cobmapintro}
Here, we briefly give the construction of cobordism maps between 
knot Floer homologies. We restrict ourselves to the case that the 
cobordism is given by a single $2$-handle attachment. We point 
the interested reader to \cite{Saha04}.\vspace{0.3cm}\\ 
Given a closed, oriented $3$-manifold $Y$ with knot $K\subset Y$ 
and a cobordism $W$ by performing a surgery along a second knot 
$L\subset Y$ which is disjoint from $K$. The surgered manifold 
will be denoted by $Y'$. These data determine a cobordism
\[
 W=[0,1]\times Y\cupb\mbox{\rm h}^{4,2}
\]
together with a canonical embedding 
$[0,1]\times\sone\hookrightarrow W$ such that
\begin{eqnarray*}
 \{0\}\times\sone\hookrightarrow K&\subset& Y\\
 \{1\}\times\sone\hookrightarrow K'&\subset& Y'.
\end{eqnarray*}
A cobordism $W$ together with such an embedding {\bf will be 
called a cobordism between $(Y,K)$ and $(Y',K')$}. Choose a 
doubly-pointed Heegaard diagram $(\Sigma,\afat,\bfat,w,z)$ which 
is adapted to $K$ and $L$ and whose pair of base points $(w,z)$ 
encode the knot $K$. Performing the surgery along $L$, by the 
constructions given by Ozsv\'{a}th and Szab\'{o}, determines a 
third set of attaching circles $\gfat=\{\gamma_1,\dots,\gamma_g\}$. 
As we know from work of Ozsv\'{a}th and Szab\'{o} the cobordism 
associated to the triple diagram $(\Sigma,\afat,\bfat,\gfat)$ is 
diffeomorphic to $W$ (see~\cite[Proposition~4.3]{OsZa03}). The information of the knot $K$ is encoded 
in the pair of base points $(w,z)$ we include into our triple diagram. 
The doubly-pointed Heegaard triple $(\Sigma,\afat,\bfat,\gfat,w,z)$ 
not only determines the cobordism $W$ but also determines the 
knots $K$ and $K'$: By definition, the Heegaard diagram 
$(\Sigma,\afat,\bfat,w,z)$ determines $(Y,K)$, the diagram $(\Sigma,\afat,\gfat,w,z)$ determines $(Y',K')$ and, observe, that the
doubly-pointed Heegaard diagram $(\Sigma,\bfat,\gfat,w,z)$ determines 
the pair
\[
  (\sthree\#^{g-1}(\stwo\times\sone),U)
\]
with $U$ the unknot. The associated knot Floer homology of the third diagram
is isomorphic to
\[
 \hfhat(\sthree\#^{g-1}(\stwo\times\sone))
\]
which admits a {\it top-dimensional} generator 
$\hattheta^+_{\bfat\gfat}$ (see~\cite[\S2.4]{OsZa03}). This generator is
uniquely represented by an intersection point $\zfat^+_{\bfat\gfat}$ in 
the diagram $(\Sigma,\bfat,\gfat)$. For intersections 
$\xfat\in\talpha\cap\tbeta$, $\yfat\in\talpha\cap\tgamma$ we define 
a Whitney triangle $\phi\co\Delta\lra\symg$ connecting $\xfat$ with $\yfat$ 
and $\zfat^+_{\bfat\gfat}$ as a continuous map from the unit disc 
in $\C$ into $\symg$ with boundary conditions like indicated in 
Figure~\ref{Fig:triangle}.
\begin{figure}[t!]
\labellist\small\hair 2pt
\pinlabel {$\yfat$} [tr] at 7 34
\pinlabel {$\xfat$} [B] at 126 235
\pinlabel {$\zfat^+$} [tl] at 247 27
\pinlabel {$\talpha$} [Br] at 71 136
\pinlabel {$\tbeta$} [Bl] at 191 136
\pinlabel {$\tgamma$} [t] at 126 26
\endlabellist
\centering
\includegraphics[width=4cm]{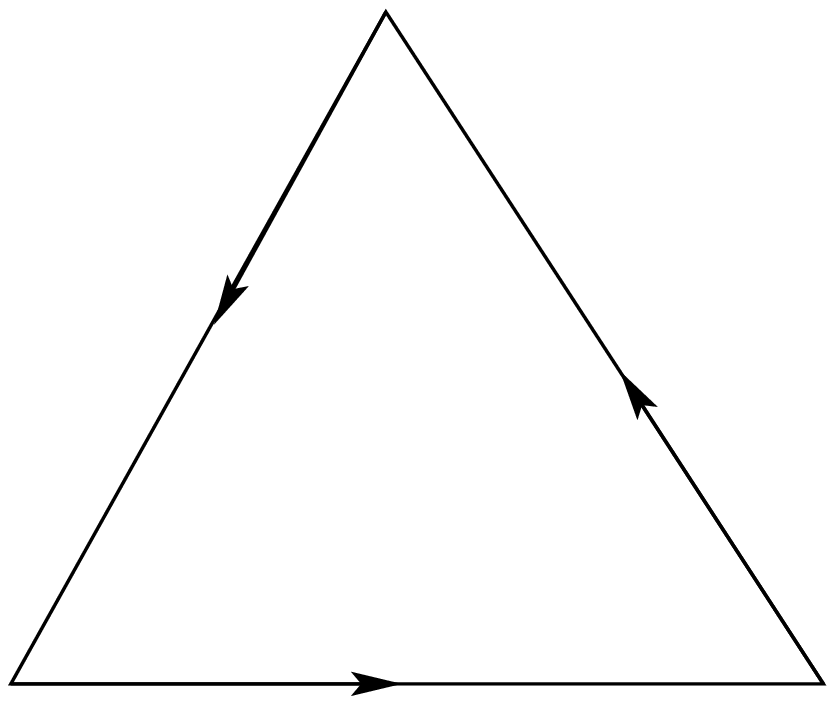}
\caption{A Whitney triangle and its boundary conditions.}
\label{Fig:triangle}
\end{figure}
We denote by 
$\mM^{\afat\bfat\gfat}(\xfat,\zfat^+_{\bfat\gfat},\yfat)$ the Maslov 
index-$1$ moduli space of holomorphic Whitney triangles connecting the 
indicated intersection points with $n_w=n_z=0$. We define
\[
 \Fhat_{\afat,\bfat\gfat}
 \co
 \cfkhat(\Sigma,\afat,\bfat,w,z)
 \lra
 \cfkhat(\Sigma,\afat,\gfat,w,z)
\]
by sending an intersection point $\xfat\in\talpha\cap\tbeta$ to
\[
 \Fhat_{\afat,\bfat\gfat}(\xfat)
 =
 \sum_{\yfat\in\talpha\cap\tgamma}
 \#\mM^{\afat\bfat\gfat}(\xfat,\zfat^+_{\bfat\gfat},\yfat)\cdot\yfat.
\]
This map descends to a map 
\[
 \Fhat_{L}
 \co
 \hfkhat(Y,K)
 \lra
 \hfkhat(Y',K')
\]
between the associated knot Floer homology
groups and just depends on the framed knot $L$ 
in an analogue way as this was proved for the maps coming 
from $2$-handle attachments
in Heegaard Floer homologies (see~\cite{Saha04}). The triple
diagram $(\Sigma,\afat,\bfat,\gfat)$ used to define the cobordism
map $\Fhat_{\afat,\bfat\gfat}$ also defines a cobordism 
$X_{\afat,\bfat,\gfat}$ (see~\cite[\S2.2]{OsZa03}) with boundary
components $-Y$, $Y'$ and $-\sthree\#^{g-1}(\stwo\times\sone)$ 
(see~\cite[\S2.2]{OsZa03}). Closing up the third boundary component
with $\#^{g-1}(B^3\times\sone)$, the cobordism is diffeomorphic
to the cobordism $W$ which is determined by the $2$-handle 
attachment along $L$ (see~\cite[Proposition~4.3]{OsZa03}). The same
way as it was done in \cite{OsZa03}, it is possible to define refinements
$\Fhat_{L;\fraks}$ of $\Fhat_L$ for $\fraks\in\spinc(W)$.\vspace{0.3cm}\\
Cobordism maps 
between knot Floer homologies
provide a surgery exact triangle like in the Heegaard Floer 
case (see \cite[Theorem 8.2]{OsZa04} and \cite[Theorem 2.7]{OsSti}). 
The proof from the Heegaard Floer case carries over verbatim to 
the knot Floer homology case.

\subsubsection{Open Books and Heegaard Diagrams}\label{sec:openbook}
We start by recalling some facts about open book decompositions of
$3$-manifolds. For details we point the reader to 
$\cite{Etnyre01}$.\vspace{0.3cm}\\
An {\bf open book} is a pair $(P,\phi)$ consisting of an oriented 
genus-$g$ surface $P$ with boundary and a homeomorphism $\phi\co P\lra P$ 
that is the identity near the boundary of $P$. The surface $P$ is called
{\bf page} and $\phi$ the {\bf monodromy}. Recall that
an open book $(P,\phi)$ gives rise to a $3$-manifold by the following 
construction: Let $c_1,\dots,c_k$ denote the boundary components of
$P$. Observe that
\begin{equation}
  (P\times[0,1])/(p,1)\sim(\phi(p),0) \label{ob:01}
\end{equation}
is a $3$-manifold with boundary given by the tori
\[
  \left((c_i\times[0,1])/(p,1)
  \sim
  (p,0)\right)
  \cong 
  c_i\times\sone.
\]
Fill in each of the holes with a full torus $\disc^2\times\sone$: we glue
a meridional disc $\disc^2\times\{\star\}$ onto 
$\{\star\}\times\sone\subset c_i\times\sone$. In this way we define 
a closed, oriented $3$-manifold $Y(P,\phi)$. Denote by $B$ the union 
of the cores of the tori $\disc^2\times\sone$. The set $B$
is called {\bf binding}. Observe that the definition of $Y(P,\phi)$ 
defines a fibration
\[
  P\hookrightarrow Y(P,\phi)\backslash B\lra\sone.
\]
Consequently, an open book gives rise to a Heegaard decomposition 
of $Y(P,\phi)$ and, thus, induces a Heegaard diagram of $Y(P,\phi)$. To see 
this we have to identify a splitting surface of $Y(P,\phi)$, i.e.~a surface 
$\Sigma$ that splits the manifold into two components. Observe that the boundary
of each fiber lies on the binding $B$. Thus
gluing together two fibers yields a closed surface $\Sigma$ of genus $2g$.
The surface $\Sigma$ obviously splits $Y(P,\phi)$ into two components and
can therefore be used to define a Heegaard 
decomposition of $Y(P,\phi)$ (cf.~\cite{HKM}).
\begin{figure}[t!]
\labellist\small\hair 2pt
\pinlabel {Page $P\!\times\!\{1/2\}$ of the open book} [bl] at 29 187
\pinlabel {$z$} [bl] at 189 112
\pinlabel {$a_i$} [t] at 76 22
\pinlabel {$b_i$} [t] at  153 22
\endlabellist
\centering
\includegraphics[height=3cm]{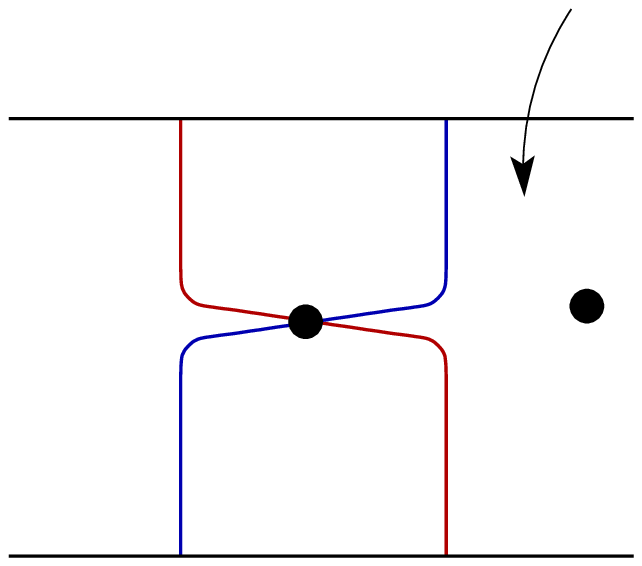}
\caption{Definition of $b_i$ and positioning of the point $z$.}
\label{Fig:zpointpos}
\end{figure}
Let $a=\{a_1,\dots,a_n\}$ be a {\bf cut system} of $P$, i.e.~a set of
disjoint properly embedded arcs such that $P\backslash\{a_1,\dots,a_n\}$
is a disc. One can easily show that being a cut system implies that
$n=2g$. Choose the splitting surface
\[
  \Sigma:=P\times\{1/2\}\cup_{\partial} (-P)\times\{1\}
\]
and let $\oa_i$ be the curve $a_i\subset P\times\{1/2\}$ with opposite 
orientation, interpreted as a curve in $(-P)\times\{0\}$.
Then define $\alpha_i:=a_i\cup\oa_i$. The curves $b_i$ are isotopic push-offs of the $a_i$. We
choose them like indicated in Figure~\ref{Fig:zpointpos}: We push the $b_i$
f the $a_i$ by following with $\partial b_i$ the positive boundary 
orientation of $\partial P$. Finally set $\beta_i:=b_i\cup\overline{\phi(b_i)}$.
The data $(\Sigma,\alpha,\beta)$ define a Heegaard diagram of $Y(P,\phi)$ 
(cf.~\cite{HKM}). 

\subsection{The Dehn Twist Sequence}\label{DTSEQ}
We will briefly recall the results given 
in \cite{Saha01}. Especially we will focus on the
derived surgery exact sequence we will call the {\bf Dehn twist sequence}.
Given an abstract open book $(P,\phi)$ and let $\delta\subset P$ be
a homologically essential simple closed curve. We try to determine how the
groups $\hfhat$-change if we compose the monodromy $\phi$ with a Dehn twist
along $\delta$, here we stick to positive Dehn twists. Let $(\Sigma,\afat,\bfat)$
be a Heegaard diagram induced by the pair $(P,\phi)$ such that $\delta$ intersects
$\beta_1$ once transversely and is disjoint from the other $\beta$-circles.
We denote by $\betaprime_1$ the curve $D_\delta^+(\beta_1)$. Note that the set of
attaching circles $\bfatprime$ which is given by
\[
  \{\betaprime_1,\beta_2,\dots,\beta_g\}
\]
and determines the manifold after the surgery. A third set of attaching circles
\[\dfat=\{\delta,\beta_2,\dots,\beta_g\}\] is formed. In the left portion of Figure~\ref{Fig:figtwo} we see 
a neighborhood of $\delta\cap\beta_1$ in the Heegaard surface (cf.~\cite{Saha01}).

\begin{figure}[t!]
\labellist\small\hair 2pt
\pinlabel {boundary of $P$} [b] at 235 328
\pinlabel {boundary of $P$} [b] at 667 328
\pinlabel {$z$} [r] at 268 250
\pinlabel {$\delta$} at 119 250
\pinlabel {$z$} [r] at 703 250
\pinlabel {$\dom_*$} [Br] at 119 212
\pinlabel {$\dom_*$} [Br] at 544 212
\pinlabel {$\dom_{**}$} [b] at 160 168
\pinlabel {$\dom_{**}$} [b] at 597 168
\pinlabel {$w$} [l] at 105 187
\pinlabel {$w$} [l] at 520 187
\pinlabel {$\dom_z$} [Bl] at 319 243
\pinlabel {$\dom_z$} [Bl] at 747 243
\pinlabel {$_2$} [l] at 348 122
\pinlabel {$_1$} [B] at 375 105
\pinlabel {$_2$} [l] at 780 122
\pinlabel {$_1$} [B] at 808 105
\pinlabel {$\beta_2$} [l] at 365 312
\pinlabel {$\alpha_2$} [l] at 365 276
\pinlabel {$\beta_1$} [l] at 365 204
\pinlabel {$\alpha_1$} [l] at 365 168
\pinlabel {$\alpha_2$} [l] at 365 60
\pinlabel {$\beta_2$} [l] at 365 25
\pinlabel {$\beta_2$} [l] at 796 312
\pinlabel {$\alpha_2$} [l] at 796 276
\pinlabel {$\betaprime_1$} [l] at 796 204
\pinlabel {$\alpha_1$} [l] at 796 168
\pinlabel {$\alpha_2$} [l] at 796 60
\pinlabel {$\beta_2$} [l] at 796 25
\endlabellist
\centering
\includegraphics[height=5cm]{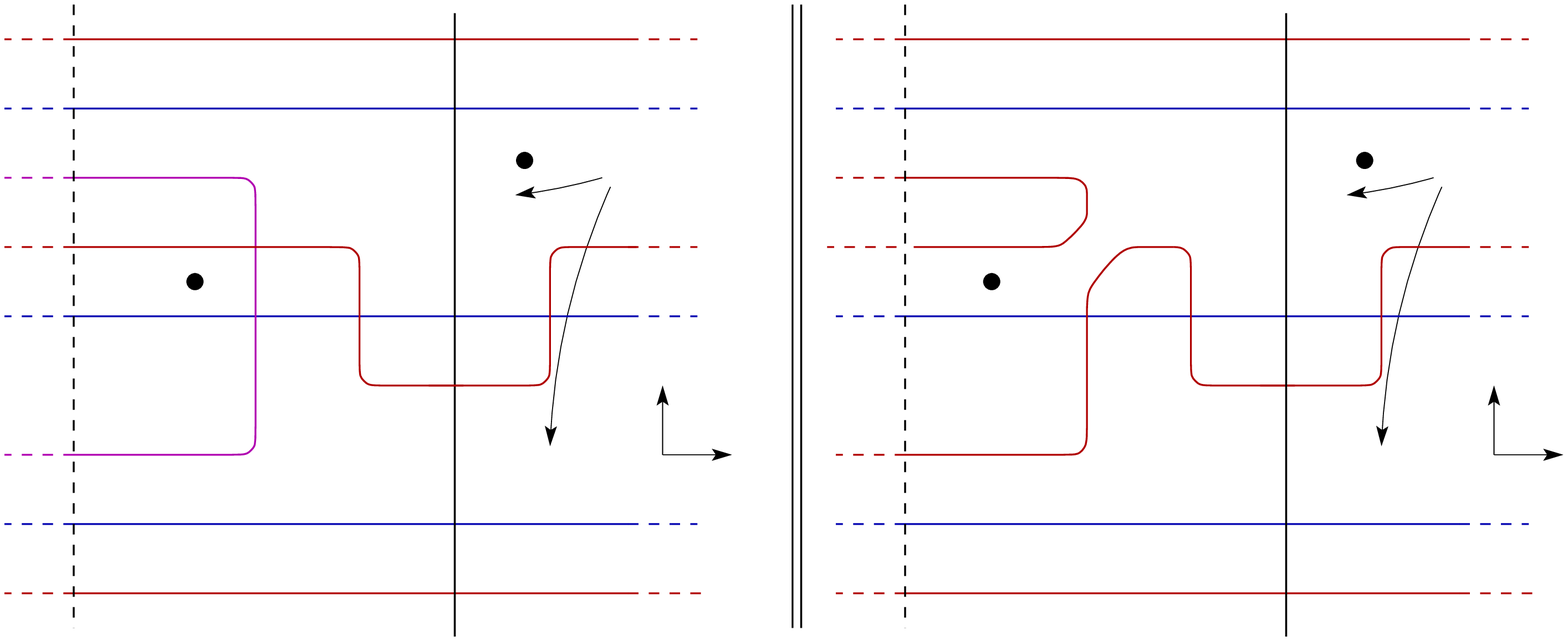}
\caption{Before and after the positive Dehn twist.}
\label{Fig:figtwo}
\end{figure}
Denote by $Y^\delta$ the manifold determined by the Heegaard diagram
$(\Sigma,\afat,\bfatprime)$. In \cite{Saha01} we have shown that in this
particular situation the homology groups $\hfhat(Y^\delta)$ can be 
interpreted as a mapping cone of the complexes $(\cfhat(\afat,\bfat),\parhat^w_{\afat,\bfat})$
and $(\cfhat(\afat,\dfat),\parhat^w_{\afat,\dfat})$ with the chain map 
$f$ given in following definition (cf.~Proposition~\ref{THMTHM}). 
\begin{definition}\label{mapdef} Define a map
\[
  f\co
  \cfkhat(\Sigma,\afat,\dfat,z,w)
  \lra
  \cfkhat(\Sigma,\afat,\bfat,z,w)
\]
by sending an element $x\in\talpha\cap\tdelta$ to
\[
  f(\xfat)
  =
  \sum_{z\in\talpha\cap\tbeta}
  \sum_{\phi\in H(\xfat,\yfat,1)}
  \modhatphi\cdot \yfat
\]
where $H(x,y,1)$ are classes in $\pitwo^{\afat\bfatprime}(x,y)$ with
$\mu=1$ and whose pair $(n_*(\phi),n_{**}(\phi))$ does not equal $(0,0)$. Here, 
$n_*(\phi)$ and $n_{**}(\phi)$ denote the multiplicities of $\phi$ at the
domains $\domstar$ and $\domststar$ (cf.~right portion of Figure~\ref{Fig:figtwo}).
\end{definition}
We would like to explain the main idea of the proof: The chain complex 
$\cfhat(\afat,\bfatprime)$ is generated by the intersection points 
$\talpha\cap\tbetaprime$. It is easy to observe that this generating set can be 
canonically identified with the disjoint union \[ 
\talpha\cap\tbeta\sqcup\talpha\cap\tdelta. \] We will call those intersections in 
$\talpha\cap\tbetaprime$ corresponding to the intersections $\talpha\cap\tbeta$ as 
{\bf $\afat\bfat$-intersections} and call the others {\bf 
$\afat\dfat$-intersections}. Due to the positioning of the point $z$ we observe 
that there is no holomorphic disc connecting an $\afat\bfat$-intersection with and 
$\afat\dfat$-intersection. Furthermore, we can identify moduli-spaces of 
holomorphic discs connecting an $\afat\bfat$-intersection with an 
$\afat\bfat$-intersection with the moduli-spaces of holomorphic discs appearing in 
the differential $\parhat^w_{\afat\bfat}$. Moreover, we can identify the 
moduli-spaces of holomorphic discs connecting an $\afat\dfat$-intersection with an 
$\afat\dfat$-intersection with the moduli-spaces of holomorphic discs appearing in 
the differential $\parhat^w_{\afat\dfat}$. Finally, there might be discs connecting 
an $\afat\dfat$-intersection with an $\afat\bfat$-intersection. We can explicitly 
characterize whose homotopy classes of Whitney discs belong to this class of discs. 
So, using this characterization we are able to define with them a chain map $f$ as 
it is done in Definition~\ref{mapdef}. By construction the associated mapping cone 
of this map $f$ is isomorphic to the Heegaard Floer homology of the chain complex 
$(\cfhat(\afat,\bfatprime),\parhat_{\afat\bfatprime})$.
\begin{prop}\label{THMTHM} Let $(\Sigma,\afat,\bfat)$ be a 
$\delta$-adapted Heegaard diagram of $Y$ and  denote by $Y^\delta$ 
the manifold obtained from $Y$ by composing the gluing map given 
by the attaching curves $\afat$, $\bfat$ with a positive 
Dehn twist along $\delta$ as indicated in Figure \ref{Fig:figtwo}. 
Then the following holds:
\[
  \hfhat(Y^\delta)
  \cong 
  H_*(\cfhat(\afat,\bfat)
  \oplus
  \cfhat(\afat,\dfat),
  \partial^{f}),
\]
where $\partial^{f}$ is of the form
\[
  \left(\begin{matrix}\parhat_{\afat\bfat}^w & 
  f\\
  0&\parhat_{\afat\dfat}^w\end{matrix}\right)
\]
with $f$ the chain map 
between $(\cfhat(\afat,\dfat),\parhat^w_{\afat\dfat})$ and 
$(\cfhat(\afat,\bfat),\parhat^w_{\afat\bfat})$ given in Definition~\ref{mapdef}.
\end{prop}
As a consequence of this fact we deduce the existence of two exact 
sequences which we call the {\bf Dehn twist sequences}
\begin{eqnarray}
&&
\begin{diagram}[size=2em,labelstyle=\scriptstyle]
 \dots & \rTo^{f_*} &\hfkhat(Y,K)&&\rTo^{\Gamma_1}&&
 \hfhat(Y_{-1}(K))&&\rTo^{\Gamma_2}&&\hfkhat(Y_0(K),\mu)&
 \rTo^{f_*}&\dots
\end{diagram}\label{dtses01}\\
&&
\begin{diagram}[size=2em,labelstyle=\scriptstyle]
 \dots & \rTo^{f_*} &\hfkhat(Y_0(K),\mu)&&
 \rTo^{\Gamma_2}&&\hfhat(Y_{+1}(K))&&\rTo^{\Gamma_1}&&
 \hfkhat(Y,K)&\rTo^{f_*}&\dots
\end{diagram}\label{dtses02}
\end{eqnarray}
where $\mu$ is a meridian of $K$ in $Y$, interpreted as sitting in $Y_0(K)$.
The Dehn twist sequences admit some invariance properties which are
similar to those of the surgery exact sequences in Heegaard Floer theory.
For details we point the reader to \cite{Saha01}.

\section{Surgery Exact Triangle and Dehn Twist Sequence}\label{sec:setadts}
The purpose of this section is to study the Dehn twist sequences and their
relationship to maps induced by cobordisms. Recall that the maps involved
in the definition of the Dehn twist sequence are not the usual cobordism
maps. Suppose we are given a closed, oriented $3$-manifold $Y$, a 
knot $K\subset Y$ and a framed knot $L$ disjoint from $K$.
We can define a map induced by a surgery along $L$ in the following way 
(cf.~\S\ref{cobmapintro}): We choose a Heegaard diagram 
$(\Sigma,\afat,\bfat)$ of $Y$ which is adapted to the 
link $K\sqcup L$. The link $K\sqcup L$ is isotopic to a two-component link
on the Heegaard surface, each of its components being a longitude of a torus component
of $\Sigma$. The knot $K$ induces a pair of points $(w,z)$ on $\Sigma$ such
that the doubly-pointed Heegaard diagram $(\Sigma,\afat,\bfat,w,z)$
encodes the pair $(Y,K)$. Performing a surgery along $L$ induces a third set of attaching circles $\gfat=\{\gamma_1\dots,\gamma_g\}$, $g$ being the
genus of $\Sigma$. The doubly-pointed Heegaard triple diagram $(\Sigma,\afat,\bfat,\gfat,w,z)$ induces a map 
\[
 \Fhat_{\afat,\bfat\gfat}
 \co
 \cfkhat(\Sigma,\afat,\bfat,w,z)
 \lra
 \cfkhat(\Sigma,\afat,\gfat,w,z)
\]
by counting holomorphic triangles with $n_z=n_w=0$ like introduced in 
\S\ref{cobmapintro}. This map descends to a map between the associated homology theories.\vspace{0.3cm}\\ 
In the following, suppose that the knots $L$ and $K$ are isotopic, 
more precisely, $K$ is a push-off of $L$ representing its framing.
In this particular situation, we can choose a Heegaard diagram which is 
adapted to both $K$ and $L$ such that $K$ and $L$ sit in the same torus component 
of $\Sigma$. This situation is somewhat special and can be realized in this
particular situation, only. Without loss of generality, we write $\Sigma$ as
$T^2\#\Sigma'$ and we may assume that both $K$ and $L$ sit in $T^2$. 
Hence, $\beta_1$ is 
a meridian of $K$ (and hence of $L$). There is a longitude $\lambda$ of $T^2$ which represents the framing of $L$. In the following,
the framing of $L$ will be our reference framing. The 
points $w$ and $z$ will lie one on each side of the curve $\beta_1$ (see 
left part of Figure~\ref{Fig:explanation}).
We perform a $(-1)$-surgery along $L$ which changes the $\beta_1$-curve to 
$\gamma_1=\beta_1+\lambda$. By choosing $\gamma_i$, $i\geq 2$ as small isotopic 
translates of the $\beta_i$, the set $\gfat=\{\gamma_1,\dots,\gamma_g\}$ is a 
Heegaard diagram of the surgered manifold. The effect of the surgery to the 
$\bfat$-circles is, that we applied to $\beta_1$ a Dehn twist about $\lambda$. 
For our purposes, it is necessary to move the point $w$ over $\lambda$, 
once (see right part of Figure~\ref{Fig:explanation}). This movement corresponds
to an isotopy of $K$ in $Y$ which makes the knot $K$ cross $L$, once. The new knot
will correspond to the $(-1)$-framing of $L$. Denote by $W$ the cobordism induced 
by the surgery. 
\begin{figure}[t!]
\labellist\small\hair 2pt
\pinlabel {$w$} [br] at 80 279
\pinlabel {$z$} [bl] at 258 279
\pinlabel {$z$} [bl] at 778 279
\pinlabel {$w$} [br] at 594 148
\pinlabel {$\lambda$} [t] at 333 132
\pinlabel {$\lambda$} [t] at 849 132
\pinlabel {$\beta_1$} [r] at 125 51
\pinlabel {$\beta_1$} [r] at 645 51
\pinlabel {isotopy of $w$} [b] at 430 219
\endlabellist
\centering
\includegraphics[height=5cm]{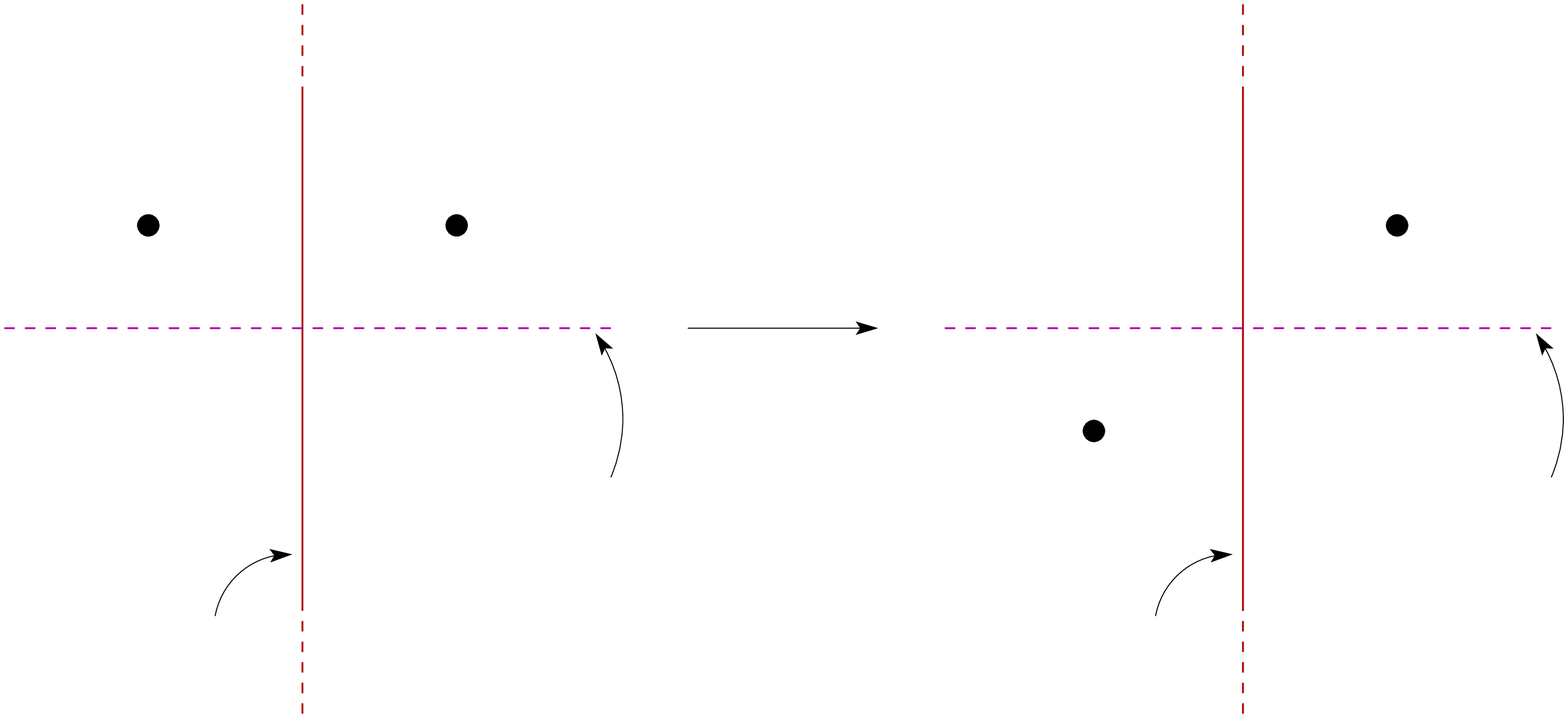}
\caption{Movement of $w$ representing an isotopy of $K$ that 
makes $K$ cross $L$, once.}
\label{Fig:explanation}
\end{figure}
We see that after the surgery the knot induced by the pair of
base points at the right of Figure~\ref{Fig:explanation} will be the unknot, i.e. we have a map
\[
 \Fhat_1\co
 \hfkhat(Y,K)
 \lra
 \hfkhat(Y_{-1}(L),U)=\hfhat(Y_{-1}(K)).
\]
We would like to show that this map sits in the exact triangle 
(see~Proposition~\ref{firsttriangle})

\[\begin{diagram}[size=2em,labelstyle=\scriptstyle] \dots& \rTo^{\partial_*}& \hfkhat(Y,K)& & \rTo^{\Fhat_{1}}& \hfhat(Y_{-1}(K))& \rTo^{\Fhat_{2}}& & \hfkhat(Y_0(K),\mu)& \rTo^{\partial_*}& \dots \\ \end{diagram}
\]
where the maps $\Fhat_i$, $i=1,2$ are defined by counting holomorphic triangles in 
suitable Heegaard triple diagrams. We will try to relate this sequence to the Dehn 
twist sequences introduced in \cite{Saha01}. To do that we will carefully analyze 
the involved Heegaard triple diagrams. Although not essential, it is opportune to 
work with Heegaard diagrams that are induced by open books. In this particular 
situation the analysis of the Heegaard diagrams and of the domains of Whitney 
triangles become easier.
So, suppose we are given an abstract open book $(P,\phi)$. This abstract open
book induces a Heegaard digram $(\Sigma,\afat,\bfat,z)$ by the algorithm
given in \S\ref{sec:openbook}. Without loss of generality we may think 
this Heegaard diagram
to be adapted to the knot $K$ (and hence adapted to $L$). Indeed, we may think
$L$ to be isotopic to a homologically essential, simple closed curve 
$\delta$ on the page $P$ of the abstract open book which intersects
$\beta_1$ once, transversely and is disjoint from the other 
$\bfat$-circles (see~\cite[Corollary 4.23]{Etnyre01} and 
\cite[Lemma 4.2]{Saha01}). We continue by defining the following sets of attaching circles
\begin{eqnarray*}
  \bfatprime
  &=&
  \{\betaprime_1,\dots,\betaprime_g\}\\
  \dfattilde
  &=&
  \{\deltatilde,\beta^{''}_2,\dots,\beta^{''}_g\},
\end{eqnarray*}
where $\betaprime_1=D_\delta^+(\beta_1)$ and $D_\delta^+$ denotes a 
positive Dehn twist about $\delta$. The $\betaprime_i$, $i\geq 2$, are 
isotopic push-offs of the $\beta_i$ such that $\beta_i$ and $\betaprime_i$ 
intersect in a cancelling pair of intersection points. Furthermore, let 
$\beta^{''}_i$, $i\geq 2$, be push-offs of the $\betaprime_i$. As above, 
the push-offs are chosen such that the $\beta^{''}_i$ and $\beta_i'$ 
intersect in a cancelling pair of intersection points. The curve 
$\deltatilde$ is given as a perturbation of the curve $\delta$, like indicated 
in Figure~\ref{Fig:atcirc}.
\begin{figure}[t!]
\labellist\small\hair 2pt
\pinlabel {$w$} [r] at 84 339
\pinlabel {$w$} [r] at 526 339
\pinlabel {$w$} [r] at 84 100
\pinlabel {$w$} [r] at 526 100
\pinlabel {$\dom_z$} [bl] at 321 386
\pinlabel {$\dom_z$} [bl] at 763 386
\pinlabel {$\dom_z$} [bl] at 321 147
\pinlabel {$\dom_z$} [bl] at 763 147
\pinlabel {$\alpha_1$} [b] at 342 308
\pinlabel {$\alpha_1$} [b] at 784 308
\pinlabel {$\alpha_1$} [b] at 342 82 
\pinlabel {$\alpha_1$} [b] at 784 82
\pinlabel {$z$} [r] at 274 403
\pinlabel {$z$} [r] at 716 403
\pinlabel {$z$} [r] at 274 164
\pinlabel {$z$} [r] at 716 164
\pinlabel {$(a)$} [b] at 102 17
\pinlabel {$(b)$} [b] at 580 17
\pinlabel {$(c)$} [b] at 102 253
\pinlabel {$(d)$} [b] at 580 253
\pinlabel {$_1$} [b] at 380 13
\pinlabel {$_1$} [b] at 822 13
\pinlabel {$_1$} [b] at 380 252
\pinlabel {$_1$} [b] at 822 252
\pinlabel {$_2$} [l] at 351 41
\pinlabel {$_2$} [l] at 793 41
\pinlabel {$_2$} [l] at 351 280
\pinlabel {$_2$} [l] at 793 280
\pinlabel {$\delta$} [t] at 58 390
\pinlabel {$\widetilde{\delta}$} [l] at 631 256
\pinlabel {$\beta_1$} [B] at 173 149
\pinlabel {$\betaprime_1$} [t] at 210 43
\pinlabel {$\widetilde{\beta}_1$} [t] at 654 43
\endlabellist
\centering
\includegraphics[height=6cm]{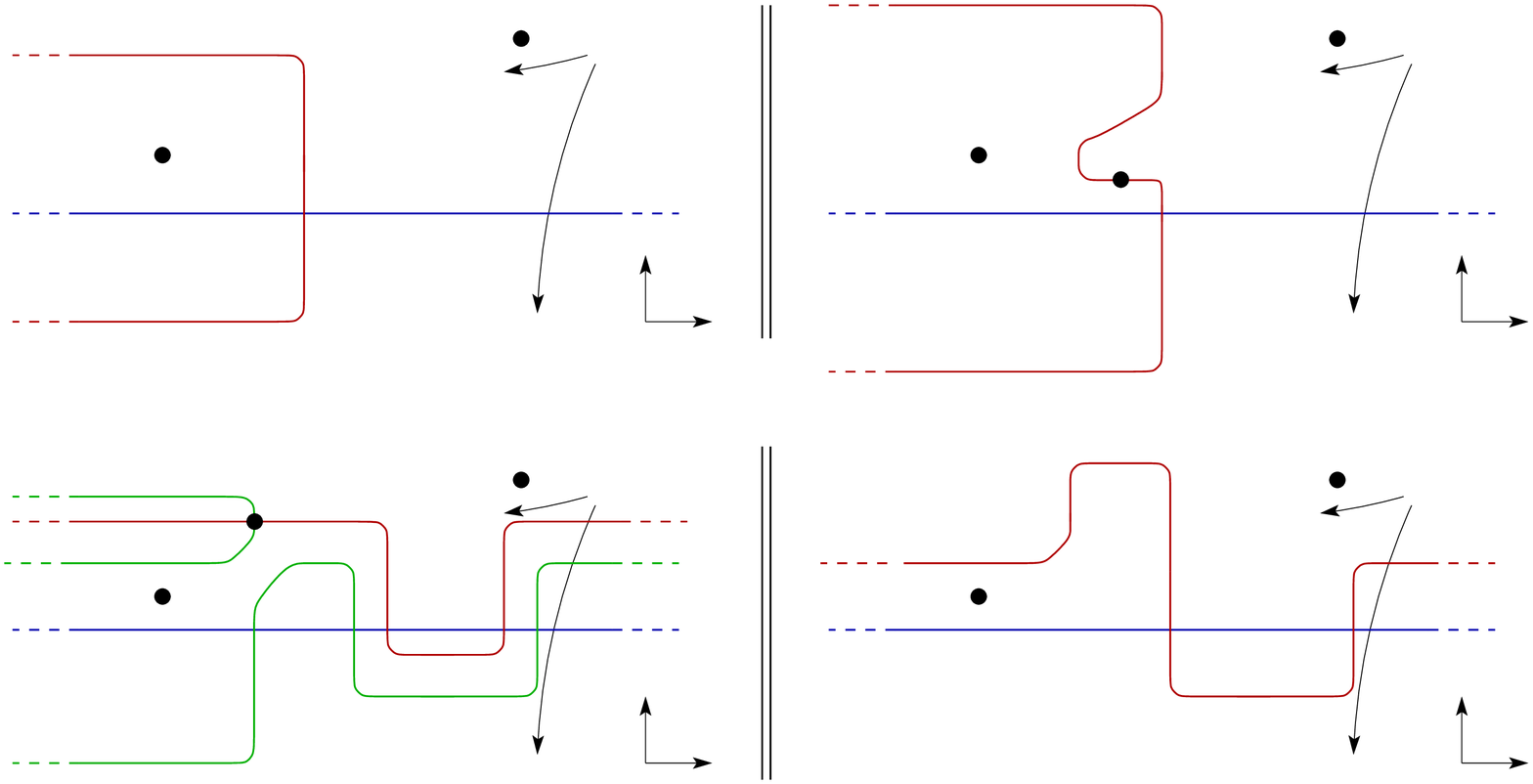}
\caption{The relevant attaching circles.}
\label{Fig:atcirc}
\end{figure}
Since the Heegaard diagram is induced by an open book, the $\alpha_1$-curve
in Figure~\ref{Fig:atcirc} is {\it surrounded } by the domain of the
base point $z$. Figure~\ref{Fig:atcirc} additionally shows the surface orientation.
We will have a close look at the following two cobordism maps:
\begin{eqnarray*}
 \Fhat_{\afat,\bfat\bfatprime}
 &\co&
 \cfkhat(\afat,\bfat,z,w)
 \lra
 \cfhat(\afat,\bfatprime,z)\\
 \Fhat_{\afat,\bfatprime\dfattilde}
 &\co&
 \cfhat(\afat,\bfatprime,z)
 \lra
 \cfkhat(\afat,\dfattilde,z,w).
\end{eqnarray*}
Since the Heegaard surface $\Sigma$ remains fixed throughout our discussion
we suppressed it from the notation. These two cobordism maps correspond to
the maps $\Fhat_1$ and $\Fhat_2$ of the Sequence~$(\ref{knotses02})$. 
Denote by $\dfat$ the set of
attaching circles $\{\delta,\betaprime_2,\dots,\betaprime_g\}$. 
By considerations from \cite{Saha01} (cf.~\S\ref{DTSEQ}) we see
that we have a short exact sequence of chain complexes.
\begin{equation}
\begin{diagram}[size=2em,labelstyle=\scriptstyle]
 0&\rTo&\cfkhat(\afat,\bfattilde,z,w)
  & \rTo^{\Gamma_1} & 
 \cfhat(\afat,\bfatprime,z) & 
 \rTo^{\Gamma_2} 
 & \cfkhat(\afat,\dfat,z,w)&\rTo0.
\end{diagram}\label{seqseq02}
\end{equation}
We see that the sequences given in (\ref{knotses02}) and (\ref{seqseq02}) 
coincide at the middle term, namely at $\cfhat(\afat,\bfatprime,z)$.
\begin{figure}[t!]
\labellist\small\hair 2pt
\pinlabel {boundary of $P$} [b] at 233 325
\pinlabel {$\beta_2$} [l] at 369 312
\pinlabel {$\alpha_2$} [l] at 369 276
\pinlabel {$\betaprime_1$} [l] at 369 204
\pinlabel {$\alpha_1$} [l] at 369 168
\pinlabel {$\alpha_2$} [l] at 369 61
\pinlabel {$\beta_2$} [l] at 369 25
\pinlabel {$\beta_1$} [B] at 171 233
\pinlabel {$z$} [r] at 273 248
\pinlabel {$\dom_z$} [l] at 321 243
\pinlabel {$w$} [l] at 93 187
\pinlabel {$_1$} [b] at 378 100
\pinlabel {$_2$} [l] at 352 125
\endlabellist
\centering
\includegraphics[height=5cm]{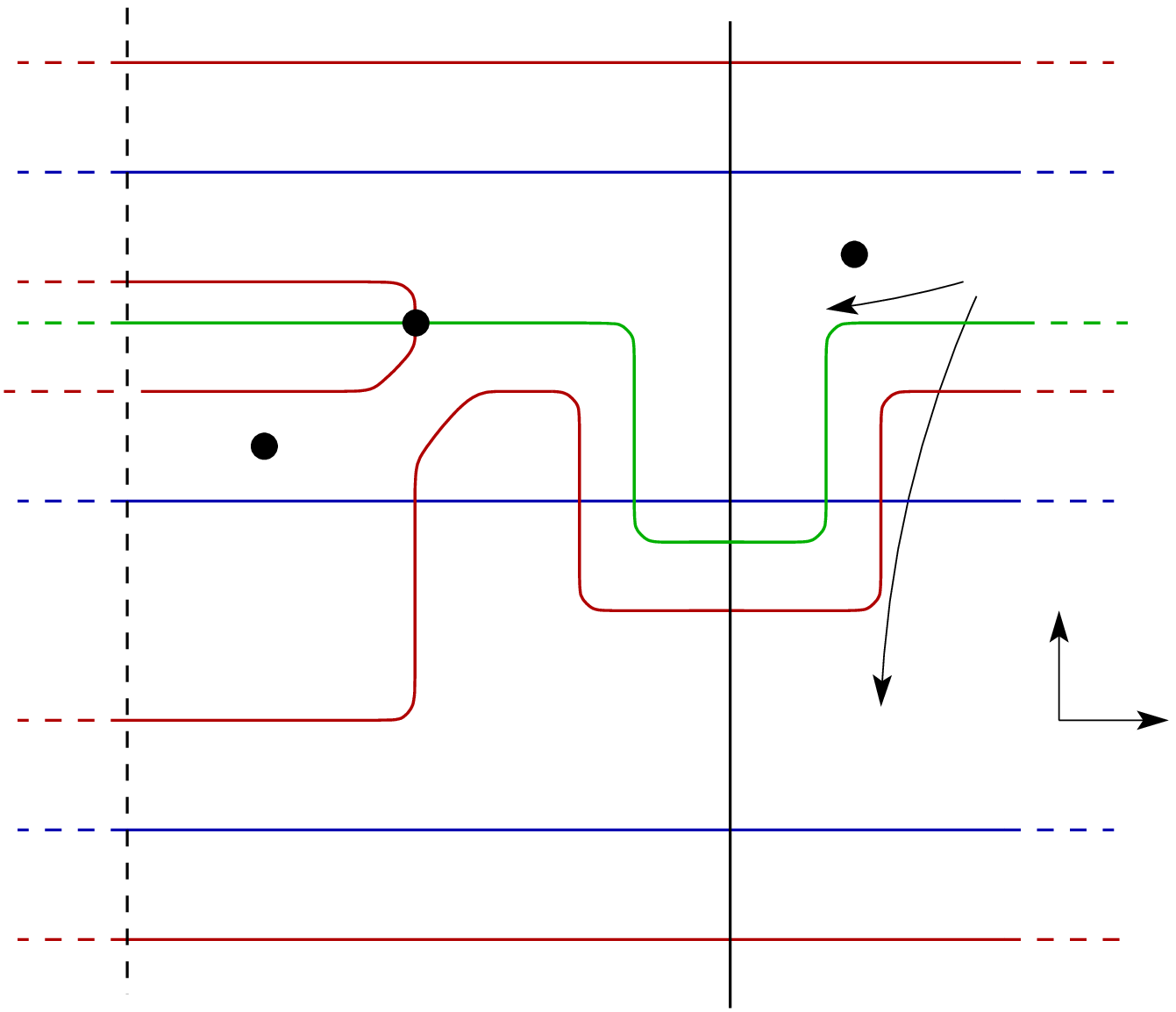}
\caption{Heegaard triple diagram for defining $\Fhat_{\afat,\bfat\bfatprime}$.}
\label{Fig:seq2-01}
\end{figure}
\begin{lem}\label{lemSES} The maps
$\Fhat_{\afat,\bfat\bfatprime}$ and $\Fhat_{\afat\bfatprime\dfat}$
respect the splitting of $\cfhat(\afat,\bfatprime,z)$, given 
in Proposition \ref{THMTHM}, i.e.~given by the Sequence~(\ref{seqseq02}).
\end{lem}
\begin{proof} We show that the claim is true for the 
map $\Fhat_{\afat,\bfat\bfatprime}$. We look at 
Figure~\ref{Fig:seq2-01} and try to show that there is 
no holomorphic triangle from an $\ab$-intersection to an 
$\ad$-intersection (cf.~\S\ref{DTSEQ}) that contributes to 
$\Fhat_{\afat,\bfat\bfatprime}$: 
\begin{figure}[t!]
\labellist\small\hair 2pt
\pinlabel {$\hattheta$} [b] at 186 168
\pinlabel {$\hattheta$} [b] at 621 168
\pinlabel {$z$} [r] at 272 172 
\pinlabel {$z$} [r] at 708 172
\pinlabel {$\dom_z$} [l] at 319 162
\pinlabel {$\dom_z$} [l] at 755 162
\pinlabel {$\beta_1$} [l] at 211 142
\pinlabel {$\beta_1$} [l] at 648 142
\pinlabel {$\betaprime_1$} [l] at 369 128
\pinlabel {$\betaprime_1$} [l] at 804 128
\pinlabel {$w$} [l] at 93 109
\pinlabel {$w$} [l] at 530 109
\pinlabel {$\alpha_1$} [l] at 369 91
\pinlabel {$\alpha_1$} [l] at 804 91
\pinlabel {$_1$} [b] at 383 22
\pinlabel {$_2$} [l] at 353 48
\pinlabel {$_1$} [b] at 820 22
\pinlabel {$_2$} [l] at 790 48
\endlabellist
\centering
\includegraphics[width=13cm]{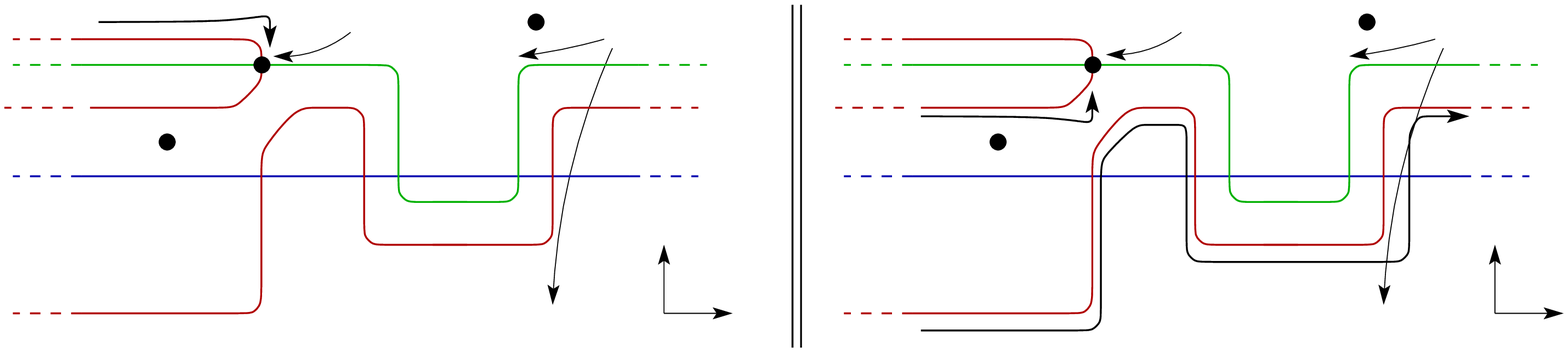}
\caption{Here we can see that $\Fhat_{\afat,\bfat\bfatprime}$ respects 
the splitting.}
\label{Fig:seq2-02}
\end{figure}

Let $\phi$ be a triangle that connects a point 
$\xfat\in\talpha\cap\tbeta$ with a point 
$\yfat\in\talpha\cap\tdelta\subset\talpha\cap\tbetaprime$. The triangle $\phi$
connects $\yfat$ with $\hattheta$ along its $\bfatprime$-boundary. 
In Figure~\ref{Fig:seq2-02}, we illustrate the two possible ways to do 
that. In both cases the $\bfatprime$-boundary of $\phi$ follows the 
black arrow pictured there. We either cause a non-negative intersection 
number $n_w$ (cf.~left of Figure~\ref{Fig:seq2-02}) or a non-negative 
intersection number $n_z$ (cf.~right part of Figure~\ref{Fig:seq2-02}). 
Thus, $n_w(\phi)\not=0$ or
$n_z(\phi)\not=0$, which shows that $\phi$ does not contribute to 
$\Fhat_{\afat,\bfat\bfatprime}$. A similar line of arguments can 
be used to prove the claim for $\Fhat_{\afat,\bfatprime\dfat}$. 
\end{proof}
As a consequence of the last lemma we see that
\[
  \Fhat_{\afat,\bfatprime\dfattilde}
  \circ
  \Fhat_{\afat,\bfat\bfatprime}=0.
\]
Indeed, we can prove the following result.
\begin{lem}\label{lemSES2} The diagram
\begin{diagram}[size=2em,labelstyle=\scriptstyle]
  \cfkhat(\afat,\bfat,z,w) &
   &
  \rTo^{\Fhat_{\afat,\bfat\bfatprime}} &
   &
  \cfhat(\afat,\bfatprime,z)\\
  \dTo^{\Fhat_{\afat,\bfat\widetilde{\bfat}}} &
   &
   &
  \ruInto(4,2)_{\iota} &\\
  \cfkhat(\afat,\widetilde{\bfat},z,w) &
   &
   &
   &\\
\end{diagram}
commutes where $\iota$ denotes the inclusion induced by a
natural identification of generators.
\end{lem}
\begin{figure}[t!]
\labellist\small\hair 2pt
\pinlabel {regions not used} [l] at 360 331
\pinlabel {by holomorphic} [l] at 360 306
\pinlabel {triangles} [l] at 360 281
\pinlabel {regions not used} [l] at 363 53
\pinlabel {by holomorphic} [l] at 363 28
\pinlabel {triangles} [l] at 363 3
\pinlabel {$\hattheta$} [l] at 168 238
\pinlabel {$z$} [r] at 273 242
\pinlabel {$z$} [r] at 714 244
\pinlabel {$\dom_z$} [b] at 322 231
\pinlabel {$\dom_z$} [b] at 764 231
\pinlabel {$\beta_1$} [b] at 502 225
\pinlabel {$\beta_1$} [l] at 211 207
\pinlabel {$w$} [l] at 93 181
\pinlabel {$w$} [l] at 535 181
\pinlabel {$\alpha_1$} [l] at 368 162
\pinlabel {$\alpha_1$} [l] at 811 162
\pinlabel {$\hattheta$} [t] at 585 138
\pinlabel {$\betaprime_1$} [t] at 213 124
\pinlabel {$\betatilde_1$} [t] at 650 124
\pinlabel {$_2$} [l] at 349 117
\pinlabel {$_2$} [l] at 791 117
\pinlabel {$_1$} [l] at 372 100
\pinlabel {$_1$} [l] at 813 100
\endlabellist
\centering
\includegraphics[width=13.5cm]{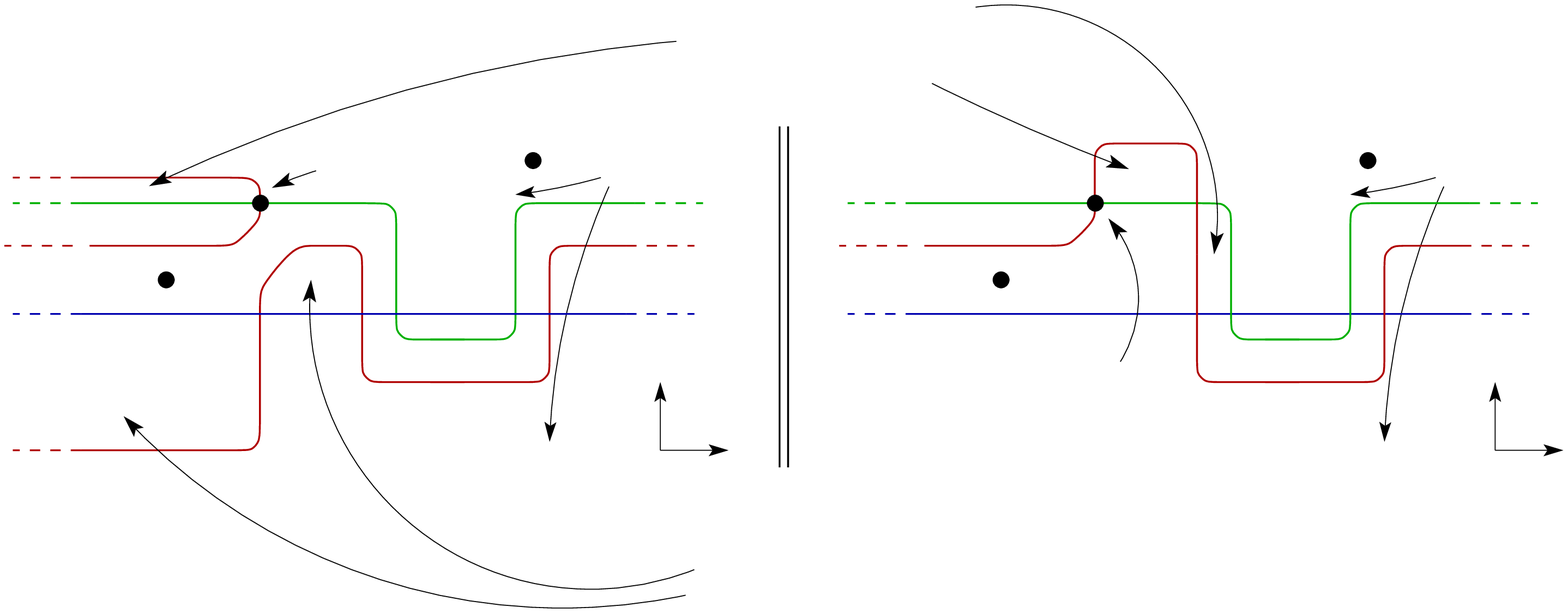}
\caption{Comparing the boundary conditions of
 $\Fhat_{\afat,\bfat\bfatprime}$ and 
  $\Fhat_{\afat,\bfat\bfattilde}$.}  
\label{Fig:seq2-03}
\end{figure}
For simplicity let us denote by 
$h$ the map $\Fhat_{\afat,\bfat\bfatprime}$ 
and by $g$ the map $\Fhat_{\afat,\bfat\widetilde{\bfat}}$. There is 
a canonical inclusion
\[
  \iota\co
  \cfkhat(\afat,\bfattilde,z,w)
  \lra
  \cfhat(\afat,\bfatprime,z,w)
\]
induced by an identification of intersection points. Namely, observe
that
\begin{eqnarray*}
  \talpha\cap\tbetaprime
  &=&
  \talpha\cap\tbeta\sqcup\talpha\cap\tdelta\\
  &=&
  \talpha\cap\mathbb{T}_{\betatilde}\sqcup\talpha\cap\tdelta
\end{eqnarray*}
in case $\bfattilde$ is a suitable perturbation of $\bfat$ we 
will define in a moment. We define $\betatilde_i=\beta_i$, for 
all $i\geq 2$, and $\betatilde_1$ as indicated in 
Figure~\ref{Fig:seq2-03} (see also Figure~\ref{Fig:atcirc}). We would 
like to show that $h=\iota\circ g$. 
\begin{definition}\label{usedomain} Let $(\Sigma,\afat,\bfat,z)$ be a 
Heegaard diagram and denote by $\dom_1,\dots,\dom_k$ the components 
of $\Sigma\backslash\{\afat\cup\bfat\}$. We say that a Whitney 
disc $\phi$ {\bf does not use} a domain $\dom_i$, $i\in\{1,\dots,k\}$, 
if the domain $\dom_i$ does not appear in $\dom(\phi)$, 
i.e.~writing $\dom(\phi)$ as 
\[
  \dom(\phi)
  =
  \sum_{j=1}^k d_{j}\cdot\dom_j,
\]
the coefficient $d_i$ vanishes. We also say that the 
domain $\dom(\phi)$ {\bf does not use} $\dom_i$.
\end{definition}
The main idea is to first prove that given 
intersections $\xfat$, $\yfat\in\talpha\cap\tbeta$, all positive 
domains $\dom$, i.e.~all coefficients in $\dom$ 
are greater than or equal to $0$, connecting $\xfat$ and $\yfat$, 
with $n_w(\dom)=n_z(\dom)=0$, do not use certain 
components of 
$\Sigma\backslash\{\afat\cup\bfat\}$ or 
$\Sigma\backslash\{\afat\cup\bfattilde\}$. 
Which domains are expected not to be used is indicated in 
Figure~\ref{Fig:seq2-03}, the left part illustrating the 
situation for $h$, the right part illustrating the situation 
for $g$. With this information, we compare the boundary conditions of 
holomorphic triangles for $h$ and $g$. The conclusion will be that, 
with its $\bfatprime$-boundary, 
the holomorphic triangles counted by $h$ always stay 
inside $\tbetaprime\cap\tbetatilde$. 
And, with their $\bfattilde$-boundary, holomorphic triangles counted 
by $g$ stay inside $\tbetatilde\cap\tbetaprime$. Thus, we are able to 
identify the moduli spaces of holomorphic triangles contributing 
to $h$ and $g$ with arguments similar to those used in the proof 
of Proposition \ref{THMTHM}.
\begin{proof} 
Figure~\ref{Fig:seq2-03} shows the part of the Heegaard triple
diagrams where the boundary conditions for the holomorphic triangles 
involved in the definition of $h$ and $g$ differ. The picture 
illustrates which regions are not used by holomorphic triangles that 
contribute to $h$ and $g$. This has to be shown in the following: 
We start our discussion with the map $h$ and look at 
Figure~\ref{Fig:proof-a}. Each part of Figure~\ref{Fig:proof-a} covers
one of the cases which we will discuss in the following. The different 
parts of Figure~\ref{Fig:proof-a} show parts of the Heegaard diagram 
pictured in the left of Figure~\ref{Fig:seq2-03}. We focused
on those parts important to our arguments. Denote by $\phi$ a holomorphic 
triangle that contributes to $h$. The domains, which we want to show not 
to be used by $\phi$, will be denoted by $\dom_{x_i}$, $i=1,2,3$. In each 
of these regions we fix a point $x_i$, $i=1,2,3$. If $\phi$ uses one of 
the domains $\dom_{x_i}$, the associated intersection number $n_{x_i}$ is 
non-zero.
\begin{figure}[t!]
\labellist\small\hair 2pt
\pinlabel {$\alpha$} [br] at 95 219
\pinlabel {$\beta$} [bl] at 130 218
\pinlabel {$\betaprime$} [t] at 113 191
\pinlabel {Interesting holomorphic} [l] at 194 241
\pinlabel {triangles} [l] at 194 216
\pinlabel {$\hattheta$} [l] at 154 189
\pinlabel {$\hattheta$} at 545 194
\pinlabel {$z$} [l] at 463 168
\pinlabel {$\beta_1$} [b] at 735 144
\pinlabel {$z$} [l] at 768 162
\pinlabel {$\beta_1$} [l] at 480 143
\pinlabel {$x_3$} [l] at 776 131
\pinlabel {$w$} [l] at 88 104
\pinlabel {$w$} [l] at 377 104
\pinlabel {$w$} [l] at 660 104
\pinlabel {$\alpha_1$} [b] at 227 88
\pinlabel {$\alpha_1$} [b] at 510 88
\pinlabel {$\alpha_1$} [b] at 794 88
\pinlabel {$x_2$} at 448 100
\pinlabel {$x_1$} at 66 55
\pinlabel {$\betaprime_1$} [b] at 36 17
\pinlabel {$\betaprime_1$} [b] at 323 17
\pinlabel {$\betaprime_1$} [b] at 602 28
\pinlabel {$_1$} [b] at 242 21
\pinlabel {$_1$} [b] at 516 23
\pinlabel {$_1$} [b] at 800 13
\pinlabel {$_2$} [l] at 214 46
\pinlabel {$_2$} [l] at 491 49
\pinlabel {$_2$} [l] at 775 44
\endlabellist
\centering
\includegraphics[width=12cm]{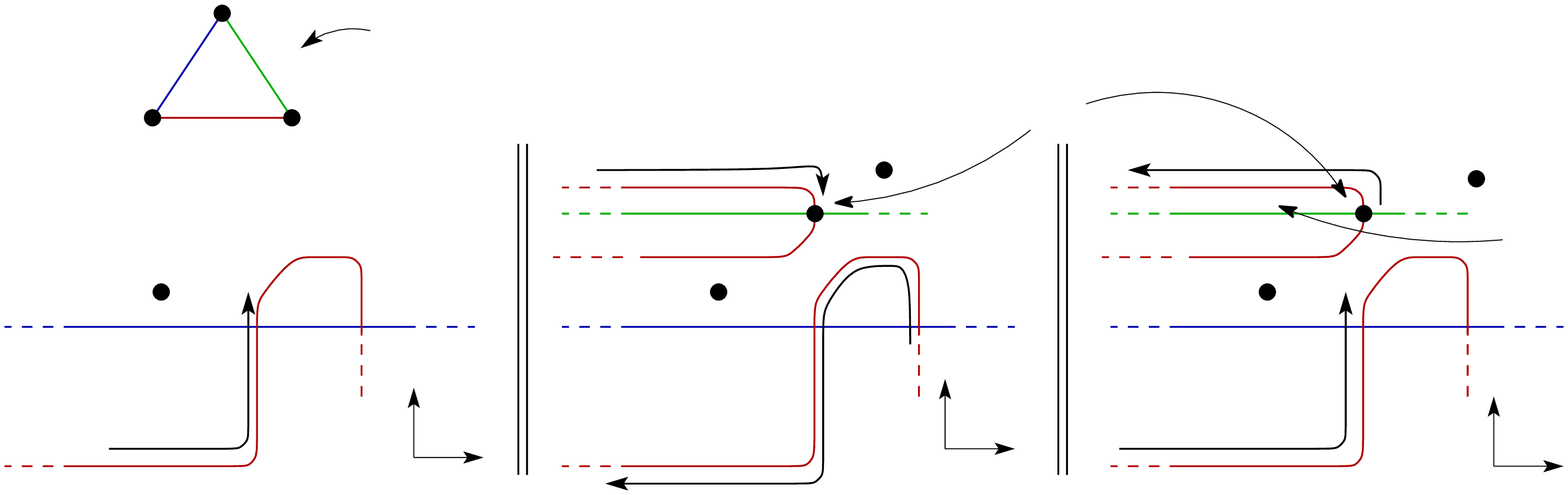}
\caption{Here we see why $n_{x_i}$, $i=1,2,3$ have to be trivial.}  
\label{Fig:proof-a}
\end{figure}
Suppose the domain $\dom(\phi)$ has
non-trivial intersection number $n_{x_1}$ (cf.~left part 
of Figure~\ref{Fig:proof-a}). This means we generate 
a $\bfatprime$-boundary pointing inside $\dom_w$, as indicated by the black arrow
in the left part of Figure~\ref{Fig:proof-a}. Consequently, 
$n_w$ has to be non-zero. Supposing the domain $\dom(\phi)$ would have
non-trivial intersection number $n_{x_2}$ (cf.~middle part 
of Figure~\ref{Fig:proof-a}), we can see from the middle part of
Figure~\ref{Fig:proof-a} (by following the black arrow), that this 
would force $n_z$ to be non-zero, since we generate 
a $\bfatprime$-boundary that has to run to $\hattheta$. 
If the domain $\dom(\phi)$ would have
non-trivial intersection number $n_{x_3}$ (cf.~right part 
of Figure~\ref{Fig:proof-a}), this would generate a 
$\bfatprime$-boundary emanating from $\hattheta$. Since $n_z$ 
vanishes, the boundary has to run once along $\betaprime_1$. But 
then $n_w$ is non-zero, as indicated by the black arrow. Hence, every 
holomorphic triangle that contributes to $h$  has trivial intersection 
number $n_{x_i}$, $i=1,2,3$.
\begin{figure}[t!]
\labellist\small\hair 2pt
\pinlabel {Interesting holomorphic} [l] at 189 260
\pinlabel {triangles} [l] at 189 235
\pinlabel {$\alpha$} [r] at 89 238
\pinlabel {$\beta$} [l] at 129 238
\pinlabel {$\betatilde$} [t] at 109 201
\pinlabel {$\hattheta$} [tl] at 147 203
\pinlabel {$\beta_1$} [B] at 60 144
\pinlabel {$x_1$} at 159 154
\pinlabel {$z$} [l] at 220 159
\pinlabel {$z$} [l] at 517 159
\pinlabel {$\beta_1$} [B] at 350 144
\pinlabel {$w$} [l] at 87 96
\pinlabel {$w$} [l] at 383 96
\pinlabel {$\alpha_1$} [b] at 235 80
\pinlabel {$\alpha_1$} [b] at 533 80
\pinlabel {$\hattheta$} [t] at 122 55
\pinlabel {$\hattheta$} [t] at 420 55
\pinlabel {$\betatilde_1$} [r] at 187 54
\pinlabel {$\betatilde_1$} [r] at 486 54
\pinlabel {$x_2$} [t] at 527 59
\pinlabel {$_1$} [b] at 92 30
\pinlabel {$_1$} [b] at 388 30
\pinlabel {$_2$} [l] at 62 55
\pinlabel {$_2$} [l] at 361 55
\endlabellist
\centering
\includegraphics[width=9cm]{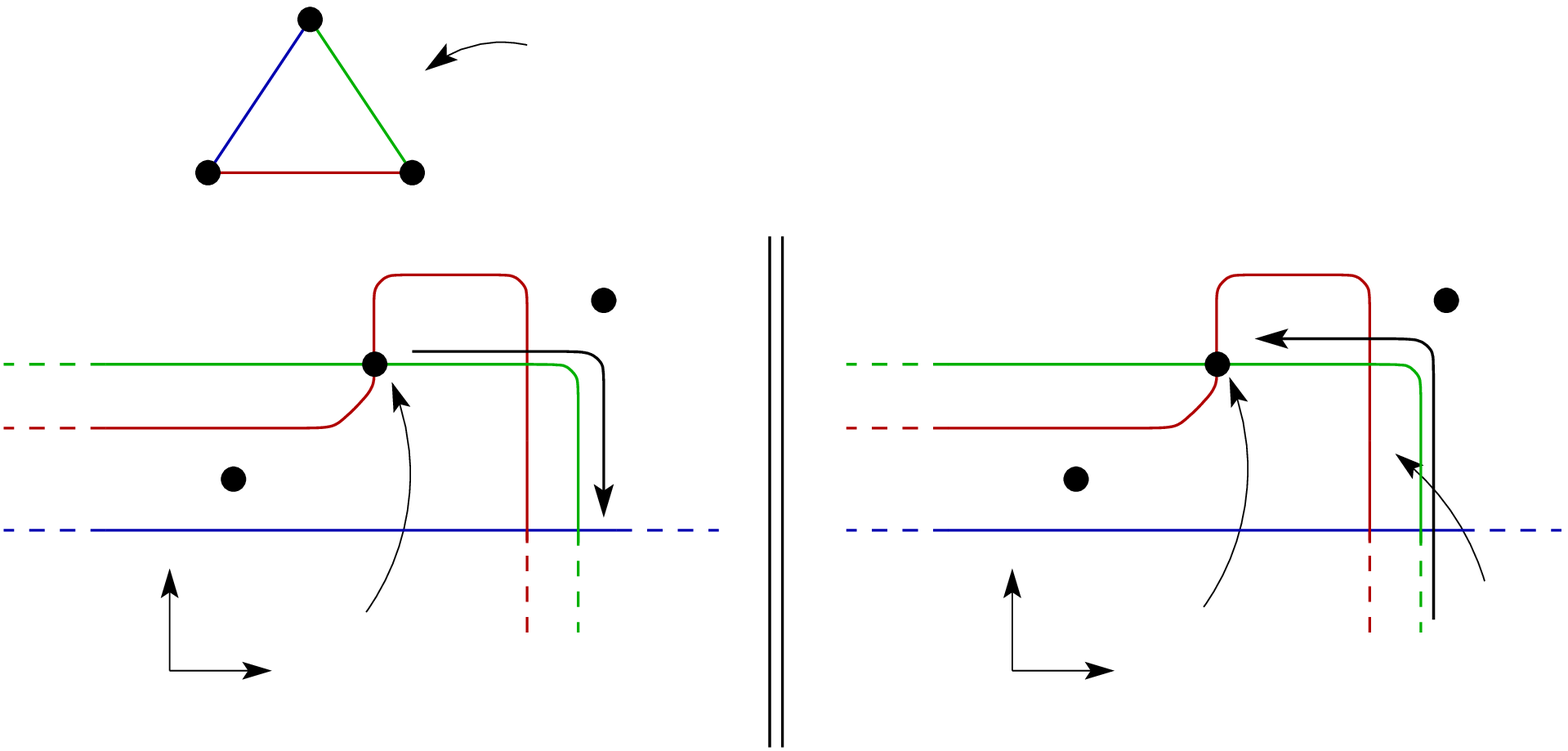}
\caption{Here we see why $n_{x_i}$, $i=1,2$ have to be trivial.}  
\label{Fig:proof-b}
\end{figure}
We continue arguing that holomorphic triangles contributing to $g$, 
cannot use the domains indicated in the right part of 
Figure~\ref{Fig:seq2-03}: Let $\phi$ be a holomorphic triangle 
contributing to $g$. Analogous to the discussion done for $h$, we 
denote the regions not expected to be used by $\phi$ with $\dom_{x_i}$, $i=1,2$.
In each of the domains we fix a point $x_i$. We want to show 
that a non-zero $n_{x_i}$ for $i\in\{1,2\}$ implies that 
$n_w\not=0$ or $n_z\not=0$. The 
different parts of Figure~\ref{Fig:proof-b} show parts of the Heegaard 
diagram pictured at the right of Figure~\ref{Fig:seq2-03}. 
Suppose the domain $\dom(\phi)$ has non-trivial intersection 
number $n_{x_1}$ (cf.~left part of Figure~\ref{Fig:proof-b}). Since $n_w=0$, we 
generate a $\bfat$-boundary pointing inside $\dom_z$, as it is indicated 
in the left part of Figure~\ref{Fig:proof-b} (the boundary follows the 
black arrow). We see that $n_z\not=0$. 
If the domain $\dom(\phi)$ would have non-trivial intersection 
number $n_{x_2}$ (cf.~right part of Figure~\ref{Fig:proof-b}), then, since
$n_z=0$, we would generate a $\bfat$-boundary pointing 
inside $\dom_w$ (cf.~right part of Figure~\ref{Fig:proof-b}). But would
imply that $n_w$ is non-zero.\vspace{0.3cm}\\
Thus, using arguments that are similar to those applied in the proof of
Proposition \ref{THMTHM}, we can identify the moduli spaces of holomorphic 
triangles that contribute to $h$ and $g$.
\end{proof}
\begin{lem}\label{lemSES3} The diagram
\begin{diagram}[size=2em,labelstyle=\scriptstyle]
   &
   &
   &
   &
  \cfkhat(\afat,\dfat,z,w)\\
   &
   &
   &
   \ruOnto(4,2)^{\pi}&
   \dTo^{\Fhat_{\afat,\dfat\widetilde{\dfat}}} \\
  \cfhat(\afat,\bfatprime,z) &
   &
  \rTo^{\Fhat_{\afat,\bfatprime\widetilde{\dfat}}} &
   &
  \cfkhat(\afat,\widetilde{\dfat},z,w)\\
\end{diagram}
commutes where $\pi$ is the projection induced by a natural
identification of generators.
\end{lem}
\begin{proof} The proof is analogous to the proof
of Lemma \ref{lemSES2}. Analogous to $\iota$ we define the 
projection $\pi$ by identifying
\begin{eqnarray*}
  \talpha\cap\tbetaprime
  &=&
  \talpha\cap\tbeta\sqcup\talpha\cap\tdelta\\
  &=&
  \talpha\cap\tbeta\sqcup\talpha\cap\mathbb{T}_{\widetilde{\dfat}},
\end{eqnarray*}
i.e.~by identifying $\talpha\cap\tdelta$ with 
$\talpha\cap\mathbb{T}_{\widetilde{\dfat}}$. This induces a 
projection $\pi$ between the respective chain modules.
\begin{figure}[t!]
\labellist\small\hair 2pt
\pinlabel {regions not used by holomorphic triangles} [l] at 183 312
\pinlabel {$\deltatilde$} [b] at 123 265
\pinlabel {$z$} [r] at 272 243
\pinlabel {$z$} [r] at 713 243
\pinlabel {$\dom_z$} [l] at 321 235
\pinlabel {$\dom_z$} [l] at 761 235
\pinlabel {$\delta$} [t] at 497 232
\pinlabel {$\hattheta$} [l] at 638 190
\pinlabel {$w$} [r] at 79 180
\pinlabel {$w$} [r] at 519 180
\pinlabel {$\alpha_1$} [B] at 347 155
\pinlabel {$\alpha_1$} [B] at 787 155
\pinlabel {$\betaprime_1$} [t] at 207 123
\pinlabel {$\deltatilde$} [l] at 631 93
\pinlabel {$\hattheta$} [l] at 169 59
\pinlabel {regions not used} [l] at 275 56
\pinlabel {by holomorphic} [l] at 275 31
\pinlabel {triangles} [l] at 275 6
\pinlabel {$_1$} [b] at 378 92
\pinlabel {$_1$} [b] at 820 92
\pinlabel {$_2$} [l] at 350 119
\pinlabel {$_2$} [l] at 792 119
\endlabellist
\centering
\includegraphics[width=13.5cm]{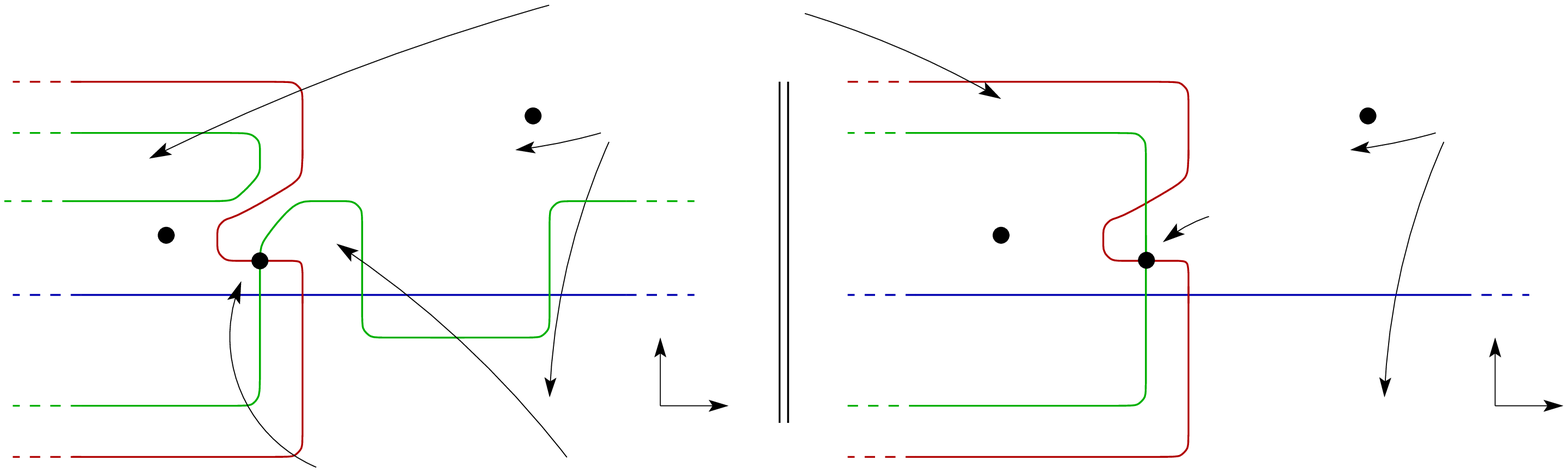}
\caption{Comparing the boundary conditions of
 $\Fhat_{\afat,\bfatprime\widetilde{\dfat}}$ and 
  $\Fhat_{\afat,\dfat\widetilde{\dfat}}$.}  
\label{Fig:seq2-03d}
\end{figure}
In the following we will denote by $h$ the map
$\Fhat_{\afat,\bfatprime\widetilde{\dfat}}$ and by $g$
the map $\Fhat_{\afat\dfat\widetilde{\dfat}}$. This time, we would 
like to show that $h=g\circ\pi$. Figure~\ref{Fig:seq2-03d} 
indicates which domains are not used by holomorphic triangles (in the sense of
Definition \ref{usedomain}) that contribute to $g$ and $h$. This has to be
shown in the following discussion: Observe, that each part of 
Figure~\ref{Fig:proof2-ab1} shows a part of the Heegaard diagrams pictured
in Figure~\ref{Fig:seq2-03d}. Each of these portions will be relevant in one
of the cases we will have to investigate. There are two domains not
to be used by holomorphic triangles contributing to $g$ (cf.~left part 
of Figure~\ref{Fig:seq2-03d}). In each of these domains we fix a point $x_i$ 
and denote the associated domain by $\dom_{x_i}$, $i=1,2$ (cf.~left and middle part 
of Figure~\ref{Fig:proof2-ab1}). There is one domain not to be used by
triangles contributing to $h$ (cf.~right part of Figure~\ref{Fig:seq2-03d}). We fix 
a point $x_3$ in this domain and denote the associated domain 
by $\dom_{x_3}$ (cf.~right of Figure~\ref{Fig:proof2-ab1}).
Let $\phi$ be a holomorphic triangle that contributes to $g$.
Suppose the domain $\dom(\phi)$ has non-trivial 
intersection number $n_{x_1}$ (cf.~left part of Figure~\ref{Fig:proof2-ab1}). 
This generates a $\bfatprime$-boundary like indicated by the black arrow in the 
left portion of Figure~\ref{Fig:proof2-ab1}. This boundary cannot be 
killed, i.e.~cannot be interpreted as sitting in the interior 
of $\dom(\phi)$, since $n_w=0$.  This $\betaprime$-boundary, thus, has to 
emanate from $\hattheta$, forcing it to follow the black arrow like 
indicated. Thus, $n_z$ is non-zero. Supposing the domain $\dom(\phi)$ would have
non-trivial intersection number $n_{x_2}$ (cf.~middle part 
of Figure~\ref{Fig:proof2-ab1}), this would create a 
$\bfatprime$-boundary like indicated by the black arrow in 
the middle portion of Figure~\ref{Fig:proof2-ab1}. This boundary 
points towards $\hattheta$. But recall that the $\bfatprime$-boundary 
of $\phi$ has to emanate from $\hattheta$, as can be seen by looking at the
triangle pictured at the top of the left and middle part of 
Figure~\ref{Fig:proof2-ab1}. Thus, we would have 
to generate a $\bfatprime$-boundary going 
along $\bfatprime$ once, completely. But this would imply $n_w$ to be 
non-zero. Now, let $\phi$ be a holomorphic triangle that contributes 
to $g$. Assuming the domain $\dom(\phi)$ would have
non-trivial intersection number $n_{x_3}$ (cf.~right part of 
Figure~\ref{Fig:proof2-ab1}), we would 
generate $\widetilde{\dfat}$-boundary 
like indicated by the black arrow in the right portion of 
Figure~\ref{Fig:proof2-ab1}. This boundary cannot be 
killed, since $n_z=0$. This boundary has to emanate 
from $\hattheta$ as can be seen by looking at the triangle pictured at 
the top of the right part of Figure~\ref{Fig:proof2-ab1}. But this is impossible, 
since $n_w=0$. We have seen that holomorphic triangles, that contribute to $h$ 
or $g$, do not use the domains indicated in Figure~\ref{Fig:seq2-03d}. Again, 
using arguments that are similar to those applied in the proof of 
Proposition~\ref{THMTHM}, we are able to identify the moduli spaces of 
holomorphic triangles that contribute to $h$ and $g$. This shows that $h=g\circ\pi$.
\end{proof}
\begin{figure}[t!]
\labellist\small\hair 2pt
\pinlabel {Interesting holomorphic} [l] at 421 280
\pinlabel {triangles} [l] at 421 255
\pinlabel {$\alpha$} [r] at 317 247
\pinlabel {$\beta$} [l] at 358 247
\pinlabel {$\alpha$} [r] at 706 247
\pinlabel {$\delta$} [l] at 748 247
\pinlabel {$\betaprime$} [t] at 338 218
\pinlabel {$\hattheta$} [l] at 381 211
\pinlabel {$\deltatilde$} [t] at 726 218
\pinlabel {$\hattheta$} [l] at 769 211
\pinlabel {$z$} [r] at 223 163
\pinlabel {$z$} [r] at 574 163
\pinlabel {$x_3$} [l] at 844 177
\pinlabel {$z$} [r] at 848 154
\pinlabel {$\deltatilde$} [l] at 112 159
\pinlabel {$\deltatilde$} [l] at 463 159
\pinlabel {$\delta$} [r] at 785 159
\pinlabel {$w$} [l] at 40 101
\pinlabel {$\alpha_1$} [b] at 288 71 
\pinlabel {$w$} [l] at 390 101
\pinlabel {$x_2$} at 469 101
\pinlabel {$\alpha_1$} [b] at 638 71
\pinlabel {$w$} [l] at 720 101
\pinlabel {$\hattheta$} [l] at 824 108 
\pinlabel {$\alpha_1$} [b] at 850 71
\pinlabel {$\hattheta$} [r] at 86 24
\pinlabel {$\betaprime_1$} [t] at 161 45
\pinlabel {$\hattheta$} [r] at 437 24
\pinlabel {$\betaprime_1$} [t] at 511 45
\pinlabel {$\deltatilde$} [l] at 808 36
\pinlabel {$_1$} [b] at 296 11
\pinlabel {$_2$} [l] at 269 37
\pinlabel {$_1$} [b] at 656 11
\pinlabel {$_2$} [l] at 629 37
\pinlabel {$_1$} [b] at 760 17
\pinlabel {$_2$} [l] at 732 41
\pinlabel {$x_1$} at 48 138
\endlabellist
\centering
\includegraphics[width=13.5cm]{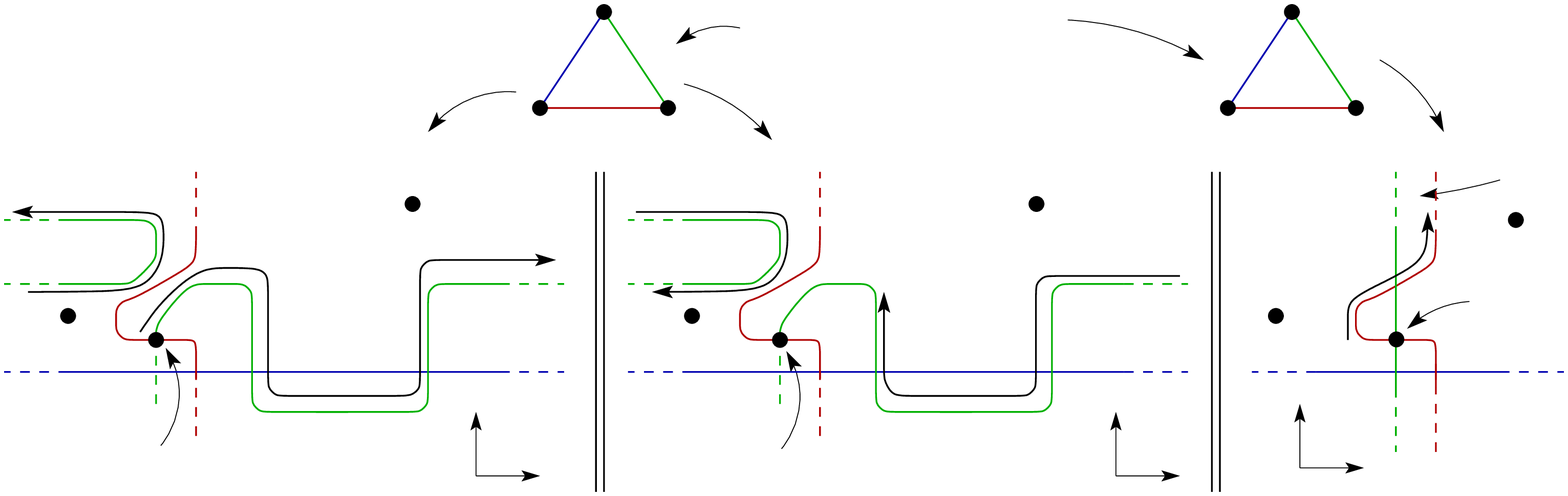}
\caption{Here, we see why $n_{x_i}$, $i=1,2$ have to be trivial for $f$
and why $n_{x_3}$ has to be trivial for $g$.}  
\label{Fig:proof2-ab1}
\end{figure}
\begin{proof}[Proof of Proposition~\ref{firsttriangle}] 
From Lemma \ref{lemSES2} and Lemma \ref{lemSES3} we conclude that 
\begin{equation}
\begin{diagram}[size=2em,labelstyle=\scriptstyle]
0&
\rTo&
\cfkhat(\afat,\bfat,z,w)& 
&
\rTo^{\Fhat_{\afat,\bfat\bfatprime}}&
\cfhat(\afat,\bfatprime,z)&
\rTo^{\Fhat_{\afat,\bfatprime\dfattilde}}& 
&
\cfkhat(\afat,\dfattilde,z,w)&
\rTo& 
0 \\
\end{diagram}
\end{equation}
is a short exact sequence of chain complexes, since (\ref{seqseq02}) is a
short exact sequence of chain complexes. Thus, by standard algebraic
topology we obtain the following long exact sequence between the 
respective homologies.
\begin{equation}
\begin{diagram}[size=2em,labelstyle=\scriptstyle]
\dots&
\rTo^{\partial_*}&
\hfkhat(Y,K)& 
&
\rTo^{\Fhat_{1}}&
\hfhat(Y_{-1}(K))&
\rTo^{\Fhat_{2}}& 
&
\hfkhat(Y_0(K),\mu)&
\rTo^{\partial_*}& 
\dots \\
\end{diagram}
\end{equation}
Observe, that $\partial_*$ is the connecting homomorphism of the short
exact sequence.
\end{proof}
We are now ready to prove the main result by combining both 
Lemma~\ref{lemSES2} and Lemma~\ref{lemSES3}. This result provides
an answer to Question~\ref{q:q}.
\begin{proof}[Proof of Theorem~\ref{ThdiagCOM}] We put together 
Lemma \ref{lemSES2} and 
Lemma \ref{lemSES3} to get two short exact sequences of
chain complexes that are related like claimed, i.e.~we have
\[
\begin{diagram}[size=2em,labelstyle=\scriptstyle]
& & & & & & & & 
\cfkhat(\afat,\dfat,z,w) & \rTo
& 0\\ 
& & & & & & 
\ruTo^{\pi} & 
& 
\dTo^{\Fhat_{\afat,\dfat\widetilde{\dfat}}}
& & \\ 
0&
\rTo&
\cfkhat(\afat,\bfat,z,w)& 
&
\rTo^{\Fhat_{\afat,\bfat\bfatprime}}&
\cfhat(\afat,\bfatprime,z)&
\rTo^{\Fhat_{\afat,\bfatprime\dfattilde}}& 
&
\cfkhat(\afat,\dfattilde,z,w)&
\rTo& 
0 \\
& & 
\dTo^{\Fhat_{\afat,\bfat\bfattilde}}& 
& 
\ruTo^{\iota} &       
& & & & & \\       
0&                 
\rTo& 
\cfkhat(\afat,\bfattilde,z,w) &                 
& & & & & & & 
\end{diagram}.
\]
To identify the diagonal sequence, i.e.~the sequence
\[
\begin{diagram}[size=2em,labelstyle=\scriptstyle]
0&
\rTo&
\cfkhat(\afat,\bfattilde,z,w)& 
&
\rTo^{\iota}&
\cfhat(\afat,\bfatprime,z)&
\rTo^{\pi}& 
&
\cfkhat(\afat,\dfat,z,w)&
\rTo& 
0
\end{diagram}
\]
with the Dehn twist Sequence~$(\ref{dtses01})$, we have 
to isotope $\betatilde_1$ a bit. Observe, that $\betatilde_1$ does 
not match with the situation presented in \S\ref{DTSEQ}. The isotopy, 
however, is supported within  $\dom_z\cup\dom_w$. Furthermore, recall that 
an isotopy not generating/cancelling intersection points, acts on the 
Heegaard Floer homology as a perturbation $\com_{s,t}$ of 
the path of almost complex 
structures $\com_{s,0}$ (see~\cite[\S 6 and \S7.3]{OsZa01} 
or cf.~\cite[\S4.2 and \S4.3]{Saha02}) used in the 
definition of the Heegaard Floer homologies. We have to see that 
the induced map $\widehat{\Phi}_{\com_{s,t}}$ (cf.~\cite[\S6]{OsZa01})
is the identity on the chain level: In the 
definition $\widehat{\Phi}_{\com_{s,t}}$ we 
count $0$-dimensional components of holomorphic discs with $n_w=n_z=0$. The
family $\com_{s,t}$ coincides with $\com_{s,0}$ outside of a set, which is
contained in $(\dom_z\cup\dom_w)\times\symgmo$, since the isotopy 
perturbing $\betatilde_1$ is supported in $\dom_z\cup\dom_w$. Thus, 
for $\xfat$, $\yfat\in\talpha\cap\tbeta$, we have an identification
\begin{equation}
  \Bigl(\M_{\com_{s,t}}(\xfat,\yfat)\Bigr)_{n_z=n_w=0}^{\mu=0}
  =
  \Bigl(\M_{\com_{s,0}}(\xfat,\yfat)\Bigr)_{n_z=n_w=0}^{\mu=0},
  \label{hohoho}
\end{equation}
where the notation should indicate that we are interested in moduli spaces 
with Maslov index $0$ and whose elements satisfy $n_z=n_w=0$. The moduli 
space on the right of Equation (\ref{hohoho}), in the following denoted 
by $\M$, is empty unless $\xfat=\yfat$: Suppose there is a holomorphic 
Whitney disc $\phi$ connecting $\xfat$ with $\yfat$. Assuming 
that $\xfat$ and $\yfat$ are not equal, the disc $\phi$ is non-constant. So, 
because of the translation action (cf.~\S\ref{prelim:01:1}) the disc $\phi$ 
comes in a $1$-dimensional family. Thus, $\phi$ cannot be an element 
of $\M$. If $\xfat$ and $\yfat$ are the same point, the moduli 
space $\M$ contains the constant holomorphic disc. But 
it does not contain non-constant holomorphic discs by the same reasoning 
done for $\xfat\not=\yfat$.\vspace{0.3cm}\\
Consequently, the map $\widehat{\Phi}_{\com_{s,t}}$ is the identity
on the chain level. We know from \cite[\S6]{OsZa01} that
the map $\widehat{\Phi}_{\com_{s,t}}$ is a chain map, i.e.~we have
\[
  0=\parhat_{\com_{s,1}}^w\circ \widehat{\Phi}_{\com_{s,t}}
  -
  \widehat{\Phi}_{\com_{s,t}}\circ\parhat_{\com_{s,0}}^w
  =
  \parhat_{\com_{s,1}}^w-\parhat_{\com_{s,0}}^w.
\]
Thus, the signed count of holomorphic discs with Maslov index $1$ in both 
\[
  \cfkhat(\afat,\bfat,z,w)
  \;\;\mbox{\rm and }
  \;\; 
  \cfkhat(\afat,\bfattilde,z,w)
\]
equals for each homotopy class admitting holomorphic representatives. Thus, we 
may replace the map $\iota$ with $\Gamma_1$. The map $\pi$ already equals $\Gamma_2$.
\end{proof}
As a consequence of these results, it is possible to refine the maps
$\Gamma_1$ and $\Gamma_2$. We would like to indicate how this is done:
Note, that the triple diagram $(\Sigma,\afat,\bfat,\bfatprime)$ comes
from a $2$-handle attachment (or, more precisely, from a $(-1)$-surgery 
along $K$). Denote by $W$ the associated cobordism. The 
map $\Fhat_{\afat,\bfat\bfat'}$ 
refines with respect to $\spinc$-structures of the 
cobordism $W$ (see~\S\ref{cobmapintro} and cf.~\cite[\S4.1]{OsZa03}). To be
precise, for $\fraks\in\spinc(W)$ denote
by $\fraks_1$ and $\fraks_2$ its restriction onto $Y$ and $Y_{-1}(K)$. Then
the map $\Fhat_1$ refines to a map
\[
 \Fhat_{1;\fraks}
 \co
 \hfkhat(Y,K;\fraks_1)
 \lra
 \hfkhat(Y_{-1}(K);\fraks_2).
\]
The map $\Fhat_4$ is given as the map induced in homology by $\Fhat_{\afat,\bfat\bfattilde}$. The underlying triple diagram $(\Sigma,\afat,\bfat,\bfattilde)$ represents the trivial cobordism, since
we obtain $\bfattilde$ from $\bfat$ by 
isotopies (see~Figure~\ref{Fig:atcirc}) which do not create/change any
intersection points with the $\afat$-circles. The $\spinc$-structures of
the trivial cobordism $[0,1]\times Y$ coincide with the 
$\spinc$-structures of $Y$. Thus, $\Fhat_4$ 
refines to
\[
 \Fhat_{4;\fraks_1}
 \co
 \hfkhat(Y,K;\fraks_1)
 \lra
 \hfkhat(Y,K;\fraks_1)
\]
for $\fraks_1\in\spinc([0,1]\times Y)=\spinc(Y)$. According to 
Theorem~\ref{ThdiagCOM}, if we define $\Gamma_{1;\fraks}$ to be the
restriction of $\Gamma_1$ onto $\hfkhat(Y,K;\fraks_1)$, then we get a map
\[
 \Gamma_{1;\fraks}
 \co
 \hfkhat(Y,K;\fraks_1)
 \lra
 \hfhat(Y_{-1}(K);\fraks_2).
\]
In a similar fashion, it is possible to define refinements $\Gamma_{2;\fraks}$ for
the map $\Gamma_2$ and refinements $f_{\fraks;*}$ for $f_*$. Given 
a $\spinc$-structure $\frakt$ of $Y\backslash\nu K$, denote by $\frakt(Y)$, 
$\frakt(Y_{-1}(K))$ and $\frakt(Y_0(K))$ the set of extensions of $\frakt$ to
$Y$, $Y_{-1}(K)$ and $Y_{0}(K)$, respectively. There is a refined version of 
the Dehn twist sequence we may derive, namely
\begin{equation}
\begin{diagram}[size=2em,labelstyle=\scriptstyle]
 \hfkhat(Y,K;\frakt(Y))&&\rTo^{\Gamma}&&\hfhat(Y_{-1}(K);\frakt(Y_{-1}(K)))\\
                       &\luTo^{\widetilde{f}}&               
                       &\ldTo_{\widetilde{\Gamma}}&\\
                       &     &\hfkhat(Y_0(K),\mu;\frakt(Y_0(K))) & &
\end{diagram}\label{eq:refine}
\end{equation}
where $\Gamma$ is the sum of $\Gamma_{1,\fraks}$ for all $\spinc$-structures
$\fraks\in\spinc(W)$ extending elements $\frakt(Y)$ and $\frakt(Y_{-1}(K))$
and, analogously, the maps $\widetilde{\Gamma}$ and $\widetilde{f}$ are defined. The
exactness of $\eqref{eq:refine}$ follows from Theorem~\ref{ThdiagCOM} and the
fact, that we may refine the horizontal sequence 
in \eqref{diagCOM} in exactly the same 
way (cf.~\cite[Theorem~14.3.2]{OzSti}).\vspace{0.3cm}\\
Similarly, applying \cite[Lemma 8.7]{OsZa01} we see that for an
orientation system $\mathfrak{o}$ on $\hfkhat(Y,K)$ there is an orientation
system $\mathfrak{o}'$ for $\hfhat(Y_{-1}(K))$ and an orientation system for $\Fhat_1$ which are {\it compatible} (in the sense of \cite[Lemma 8.7]{OsZa01}), i.e.~the 
map $\Fhat_1$ can be defined with $\Z$-coefficients.
The same we do with $\Fhat_4$, however, as above, using that the associated
cobordism is $[0,1]\times Y$. We obtain that $\Fhat_4$ can be defined with
$\Z$-coefficients for the orientation system $\mathfrak{o}$ on the source
and target $\hfkhat(Y,K)$. Thus, using the commutativity of the diagram in
Theorem~\ref{ThdiagCOM}, or using the chain level version derived in the proof
of Theorem~\ref{ThdiagCOM}, we may give a version of $\Gamma_1$ 
with $\Z$-coefficients.
\section{Proofs of Theorem~\ref{kerimequal} and 
Proposition~\ref{combinmap}}\label{gendef}
Suppose we are given a closed, oriented $3$-manifold $Y$ and a
framed knot $L\subset Y$ with framing we denote by $n$. Denote 
by $K$ a push-off of $L$ corresponding to the $(n+1)$-framing 
of $L$. The goal of the following discussion will be to
associate to a $n$-surgery along $L$ a connecting morphism $f_*$
 (see~Definition~\ref{mapdef} and cf.~Proposition~\ref{THMTHM}) from 
a suitable Dehn twist sequence. Denote by $(P,\phi)$ an open book
decomposition of $Y'=Y_n(L)$ which is adapted to $K'\sqcup L'$ (in the 
sense of \S\ref{sec:setadts}). Here, $L'$ denotes
the knot $K$ represents in the surgered manifold $Y_n(L)$ with framing given by
$n+2$ and $K'$ is a push-off of $L'$ representing its $n+1$-framing. 
Denote by $(\Sigma,\afat,\bfat,w,z)$ the doubly-pointed Heegaard
diagram induced by $(P,\phi)$ like constructed in \S\ref{sec:setadts}. 
Using the notation from
that section, we would like to show that the diagram 
$(\Sigma,\afat,\dfat,w,z)$ is a Heegaard diagram of $Y$ which is 
adapted to the knot $K$. If we are able to show this, then the Dehn twist
sequence will read
\begin{equation}
\begin{diagram}[size=2em,labelstyle=\scriptstyle]
 \hfkhat(\Sigma,\alphafat,\deltafat,z,w) 
 & & \rTo^{f_*} & & 
 \hfkhat(\Sigma,\alphafat,\betafat,z,w) \\
 &\luTo_{\Gamma_2}& 
 &\ldTo_{\Gamma_1}& \\
 && 
 \hfhat(\Sigma,\alphafat,\betaprimefat,z)
 &&
\end{diagram}\label{seqseq03}
\end{equation}
and we see, that the connecting morphism $f_*$ in this Dehn twist sequence is
a morphism between the knot Floer homologies of the pair 
$(Y,K)$ and of the pair $(Y_n(L),K')$: To prove the claim, we will give a
surgical description of the manifold represented by $(\Sigma,\afat,\dfat,w,z)$.
Using the notation from \S\ref{sec:setadts} we know that the Heegaard diagram
$(\Sigma,\afat,\dfat,w,z)$ represents the pair $(Y'_{n+2}(L'),\mu)$, where $\mu$
is a meridian of $L'$ in $Y'$, interpreted as sitting in $Y'_{n+2}(L')$. Just
note, that in \S\ref{sec:setadts} we measured the surgery coefficients with
respect to the page framing of $L'$ induced by $(P,\phi)$. Thus, the $0$-framing
from $\S\ref{sec:setadts}$ corresponds to the specified surgery framing of $L'$, 
i.e.~the $(n+2)$-framing. We obtain the surgery description of
$(Y'_{n+2}(L'),\mu)$ given in the left of Figure~\ref{Fig:fdef}. We slide the knot
$\mu$ over $L'$, once, and obtain the middle portion of Figure~\ref{Fig:fdef}. 
Finally, we perform an inverse handle slide of $L'$ over $L$ to obtain the
right portion of Figure~\ref{Fig:fdef}. With a slam dunk, we can slide $L'$
off $L$ which makes both surgery curves disappear in the surgery diagram. 
The resulting manifold is $Y$ (as the surgeries we did now disappeared) and, since
$\mu$ is a parallel of $L$ (see right of Figure~\ref{Fig:fdef}), the 
knot $\mu$ equals the knot $K$.\vspace{0.3cm}\\
Thus, the connecting morphism which can be defined as in Definition~\ref{mapdef}
induces a map
\[
  f_{*}
  \co
  \hfkhat(Y,K)
  \lra
  \hfkhat
  (Y_n(L),K').
\]
Before we continue with the proofs of the main results of this section, we 
would like to make the following observation.
\begin{figure}[t!]
\labellist\small\hair 2pt
\pinlabel {$n+1$} at 108 273
\pinlabel {$n+1$} at 415 273
\pinlabel {$n+1$} at 756 273

\pinlabel {$L'$} [t] at 66 30
\pinlabel {$L'$} [t] at 413 30
\pinlabel {$L$} [t] at 151 30
\pinlabel {$L$} [t] at 458 30
\pinlabel {$L$} [t] at 800 30

\pinlabel {$\mu$} [t] at 372 30
\pinlabel {$\mu$} [t] at 715 30

\pinlabel {slide} [b] at 253 221
\pinlabel {inverse} [b] at 568 221
\pinlabel {handle slide} [t] at 568 200

\pinlabel {$\mu$} [r] at 23 158
\pinlabel {$L'$} [l] at 838 158
\pinlabel {$0$} [l] at 833 198

\pinlabel{$n+2$} [b] at 415 433
\pinlabel{$n+2$} [r] at 60 381
\pinlabel{$n$} [l] at 156 381
\pinlabel{$n$} [l] at 463 381
\pinlabel{$n$} [l] at 806 381

\endlabellist
\centering
\includegraphics[width=12cm]{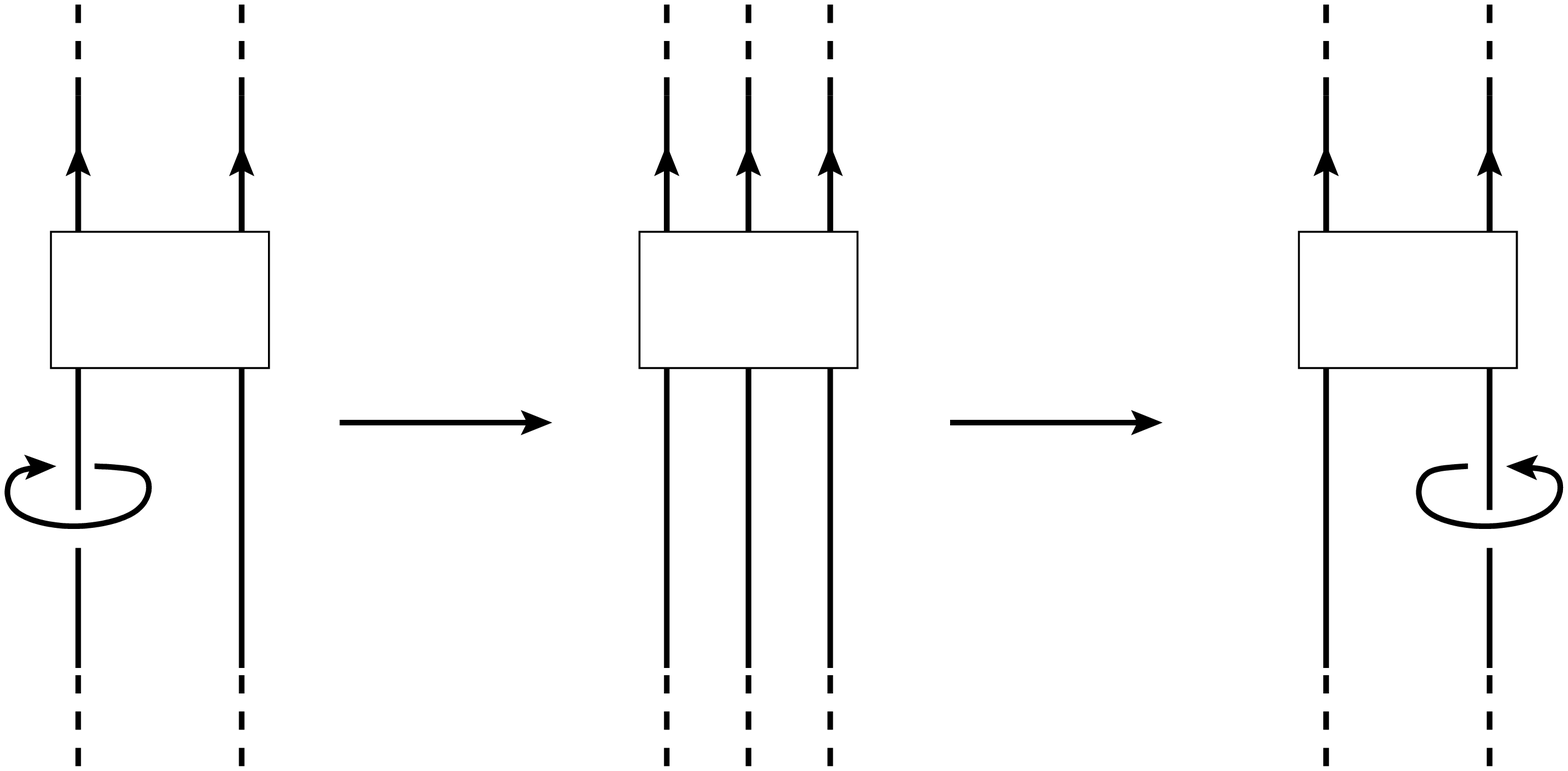}
\caption{A surgery diagram for the pair $(Y_{n+2}(L),\mu)$ and the moves 
necessary to see that $(Y_{n+2}(L),\mu)=(Y,K)$.}
\label{Fig:fdef}
\end{figure}
\begin{lem}\label{observe} In Theorem~1.1 we can drop the condition 
that $K\subset Y\backslash C$ is null-homologous.
\end{lem}
\begin{proof}[Sketch of Proof] We go through the mapping cone proof of the
surgery exact triangle (see~\cite[Theorem~4.5]{OsZa05}, cf.~also the end of the 
proof of Proposition~2.1~in \cite{OsZa05} and cf.~\cite[\S7]{Saha02}) and 
see that we do not need the condition.
\end{proof}
\begin{proof}[Proof of Theorem~\ref{kerimequal}]
Suppose we are given a manifold $Y'$ with knot $K'$ in it with framing $n$. Denote
by $Y$ the result of an $n$-surgery along $K'$ and denote by $K'$ the knot induced 
from the push-off of $K'$ corresponding to the $(n-1)$-framing of $K'$. From 
the considerations given in this section we know that to this situation we may 
associate a map $f_*$ (defined as in Definition~\ref{mapdef}) which
is part of a Dehn twist sequence, i.e.~we may set up a Dehn twist sequence 
which reads
\begin{diagram}[size=2em,labelstyle=\scriptstyle]
 \hfkhat(Y',K') &     & \rTo^{f_*}     &     & \hfkhat(Y,K) \\
              &\luTo_{\Gamma_2} &          & \ldTo_{\Gamma_1}& \\
              &     &\hfhat(Y'')&     & \\
\end{diagram}
for some manifold $Y''$. In this situation, $Y''$ will correspond to
$Y_{-1}(K)$ and the pair $(Y',K')$ will equal $(Y_0(K),\mu)$.\vspace{0.3cm}\\
On the other hand, we may apply Proposition~\ref{firsttriangle} to the 
pair $(Y,K)$ to get the following exact sequence
\begin{equation}
\begin{diagram}[size=2em,labelstyle=\scriptstyle]
 \hfkhat(Y',K') &     & \rTo^{\partial_*}     &     & \hfkhat(Y,K) \\
              &\luTo_{\Fhat_2} &          & \ldTo_{\Fhat_1}& \\
              &     &\hfhat(Y'')&     & \\
\end{diagram}
\label{seq:upper}
\end{equation}
Furthermore, by Theorem~1.1 (and the observation formulated in 
Lemma~\ref{observe}) we see that the pair $(Y,K)$ induces a surgery
exact triangle which reads
\begin{equation}
\begin{diagram}[size=2em,labelstyle=\scriptstyle]
 \hfkhat(Y',K') &     & \rTo^{\Fhat_3}     &     & \hfkhat(Y,K) \\
              &\luTo_{\Fhat_2} &          & \ldTo_{\Fhat_1}& \\
              &     &\hfhat(Y'')&     & \\
\end{diagram}
\label{seq:lower}
\end{equation}
Since the maps $\Fhat_i$, $i=1,2$, appear in both the Sequences~\eqref{seq:upper}
and \eqref{seq:lower}, by exactness we conclude that
\begin{equation}
\begin{array}{ccc}
 \ker(\Fhat_3)&=&\ker(\partial_*)\\
 \im(\Fhat_3)&=&\im(\partial_*).
\end{array}
\label{kerimeq}
\end{equation}
Now, we continue going back to the situation of Theorem~\ref{ThdiagCOM}, especially
considering the notations used there. We consider
the commutative diagram
\begin{equation}
\begin{diagram}[size=2em,labelstyle=\scriptstyle]
 \hfkhat(Y,K) & &
 \rTo^{f_*} 
 & &\hfkhat(Y_0(K),\mu)\\
 \dTo_{\Fhat_4} && &&\uTo_{\Fhat_5} \\
 \hfkhat(Y,K) & &
 \rTo^{\partial_*}
 & &\hfkhat(Y_0(K),\mu)
\end{diagram}\label{important}
\end{equation}
where $\Fhat_4$ is the map in homology induced 
by $\Fhat_{\afat,\bfat\bfattilde}$ and $\Fhat_5$ is the map in homology 
induced by $\Fhat_{\afat,\dfat\dfattilde}$. This diagram
is the square from Sequence~(\ref{diagCOM}) and which commutes according 
to Theorem~\ref{ThdiagCOM}. As we have observed in (\ref{kerimeq}), the kernel 
and the image of $\Fhat_{3}$ coincides with the kernel and image 
of $\partial_*$. Thus, we have that
\begin{eqnarray*}
 \im(f_*)
 &=&\im(\Fhat_5\circ\partial_*\circ \Fhat_4)
 =\im(\Fhat_5\circ\Fhat_3\circ \Fhat_4)
 \\
 \ker(f_*)
 &=&\ker(\Fhat_5\circ\partial_*\circ \Fhat_4)
 =\ker(\Fhat_5\circ\Fhat_3\circ \Fhat_4)
\end{eqnarray*}
Observe, that the map $\Fhat_3$ is given as the map in homology
induced by the Heegaard triple diagram $(\Sigma,\afat,\dfattilde,\bfat)$.
The composition $\Fhat_5\circ\Fhat_3\circ \Fhat_4$ equals
\begin{eqnarray*}
\Fhat_{\afat,\bfat\bfattilde}
\circ
\Fhat_{\afat,\dfattilde\bfat}
\circ
\Fhat_{\afat,\dfat\dfattilde}
=
\Fhat_{\afat,\bfat\bfattilde}
\circ
\Fhat_{\afat,\dfat\bfat}
=
\Fhat_{\afat,\dfat\bfattilde},
\end{eqnarray*}
where we used the composition law for cobordism maps and the fact 
that $\Fhat_3$ is the map, which is induced 
by $\Fhat_{\afat,\dfattilde\bfat}$ in homology. Hence, the image 
and kernel of the map $f_*$ coincides with the image and
kernel of the map induced by $\Fhat_{\afat,\dfat\bfattilde}$ in homology.
However, the maps $\Fhat_3$ and $\Fhat_{\afat,\dfat\bfattilde}$ are induced
by the same $2$-handle attachment, which can be seen by comparing the
Heegaard triple diagrams $(\Sigma,\afat,\dfattilde,\bfat)$ 
and $(\Sigma,\afat,\dfat,\bfattilde)$.
\end{proof}
\begin{proof}[Proof of Proposition~\ref{combinmap}] 
Let $\Fhat$ be the map given in the statement of the proposition. By the
discussion at the beginning of this section, we know that to the described
surgery we may associate a map $f_*$ (as defined in Definition~\ref{mapdef}) 
which is part of a Dehn twist sequence. Furthermore, by Theorem~\ref{kerimequal}
we know that the rank of the image and the kernel of $\Fhat$ and $f_*$ agree.
Thus, to prove the statement, we have to give a reasoning why the kernel
and image of $f_*$ can be computed, combinatorially:\vspace{0.3cm}\\
Using the notation from the beginning of this section, we have to prove
that it is possible to find an open book decomposition $(P,\phi)$ adapted
to $L'$ and $K'$ such that the induced Heegaard diagram $(\Sigma,\afat,\bfat)$
is nice (see~\cite{sarwang}). In \cite{plamenev}, Plamenevskaya shows that 
the Sarkar-Wang algorithm (see \cite{sarwang}) can 
be modified to apply for open books by just using isotopies 
of the monodromy: To be more precise, it is possible to modify the monodromy 
$\phi$ with a suitable isotopy $\varphi_t$ such that the open book $(P,\phi')$,
where $\phi'=\varphi_1\circ\phi$, induces a nice Heegaard diagram. We apply 
her modified version of the Sarkar-Wang algorithm to obtain such an isotopy 
$\varphi_t$. The knot $K'$ (and respectively $L'$) can be modified with the
isotopy $\varphi_t$, as well. The knot $K'$ is isotopic to a curve $\delta$
on the page $P$ of the open book (since it is an adapted open book 
by definition). The isotopy deforms $\delta$ into $\varphi_1(\delta)$. The 
open book $(P,\phi')$ is an open book decomposition adapted to 
$\varphi_1(\delta)$. To complete
the proof, we have to see that $(\Sigma,\afat,\bfatprime)$ is nice, as well: 
To see this, recall (for instance from \cite[proof of Lemma~4.2]{Saha01}) 
that $\delta$ (on the 
Heegaard surface) is parallel to $\beta_2$ outside of the region pictured in
Figure~\ref{Fig:figtwo} (cf.~Figure~\ref{Fig:atcirc}). Hence, the domains of 
the diagram $(\Sigma,\afat,\bfatprime)$ are, besides $\dom_z$, all 
obtained from the domains of $(\Sigma,\afat,\bfat)$ by splitting 
off a rectangle. The domain of the point $z$ in the new diagram is 
obtained by joining together $\dom_z$ and $\dom_w$ of the old diagram
(see~Figure~\ref{Fig:figtwo}). Thus, niceness is preserved. As 
$f_*$, by definition, is a part of the differential 
$\parhat_{\afat\bfatprime}$ (see~Proposition~\ref{THMTHM}), its image and 
kernel can be now computed combinatorially.
\end{proof}

\addcontentsline{toc}{chapter}{Bibliography}
\bibliographystyle{amsplain}
\providecommand{\bysame}{\leavevmode\hbox to3em{\hrulefill}\thinspace}
\providecommand{\MR}{\relax\ifhmode\unskip\space\fi MR }
\providecommand{\MRhref}[2]{%
  \href{http://www.ams.org/mathscinet-getitem?mr=#1}{#2}
}
\providecommand{\href}[2]{#2}

\end{document}